\documentclass[a4paper,12pt,reqno]{amsart}
\usepackage{amsmath,amssymb,amsthm, color}
\numberwithin{equation}{section}
\newtheorem{theorem}{Theorem}[section]

\newtheorem{proposition}[theorem]{Proposition}
\newtheorem{corollary}[theorem]{Corollary}
\newtheorem{lemma}[theorem]{Lemma}
\newtheorem*{theoremA}{Theorem A}
\newtheorem*{theoremB}{Theorem B}
\theoremstyle{definition}
\newtheorem*{remark}{Remark}

\newcommand{\vep}{\varepsilon}
\newcommand{\R}{{\bf R}}

\newcommand{\dsp}{\displaystyle}
\newcommand{\lam}{\lambda}
\begin{document}

\title
{Infinite time blow-up of solutions to the classical 
Keller-Segel model of chemotaxis in higher dimensions}

\author{Y\={u}ki Naito}
\address[Y\={u}ki Naito]{
Department of Mathematics,  Hiroshima University \\
Higashi-Hiroshima, 739-8526, Japan}
\email{yunaito@hiroshima-u.ac.jp}

\author{Takasi Senba}
\address[Takasi Senba]{Faculty of Engineering, Kanagawa University \\
3-27-1 Rokkakubashi, Kanagawa-ku, Yokohama 221-8686, Japan}
\email{tsenba3382@kanagawa-u.ac.jp}

\date{\today}

\begin{abstract}
We study the simplest parabolic-elliptic model of chemotaxis in $\R^N$, and   
show the optimal criteria for the existence of infinite time blow-up of solutions with 
the initial data below the Chandrasekhar singular solution when $N \geq 10$. 
In particular, we exhibit the differences of the criteria between the cases $N \geq 11$ and $N = 10$. 
Our argument is based on the study of the Cauchy problem 
for the transformed equation involving the averaged mass of the solution and 
the asymptotic expansions of solutions to the Liouville equation.
\end{abstract}



\maketitle


\section{Introduction}

We consider the following Cauchy problem for the parabolic-elliptic system 
in space dimensions $N \geq 3$:
\begin{equation}
	\left\{
	\begin{array}{ll}
	u_t = \nabla\cdot (\nabla u - u\nabla v), 
	\quad & x \in \R^N, \ t > 0,
	\\[1ex]
	0 = \Delta v + u, 
	\quad & x \in \R^N, \ t > 0,
	\\[1ex]
	u(x, 0) = u_0(x) \geq 0, \quad &x \in \R^N.
	\end{array}
	\right.
	\label{eq1.1}
\end{equation}
This system arises as a simplified model of chemotaxis processes, 
where $u$ and $v$ respectively denote the density of the bacterial population 
and the concentration of the secreted chemoattractant. 
The system (\ref{eq1.1}) is also known as a model describing gravitational interactions of particles. 
(See, e.g., \cite{BBTW, Bilc, Cha, KeSe}. ) 

The dynamics of the system and its analysis strongly depend on the space dimension $N$. 
When $N = 2$, it is well known that an 
interesting $8\pi$ critical mass phenomenon has been well studied for the global and 
blow-up solutions (see, e.g. \cite{BeMa, BKLNa, BKLNb, BCM, BDP}). 
For the case $N \geq 3$, the system (\ref{eq1.1}) admits both globally bounded solutions and blow up 
solutions, but the critical mass does not exist
(see, e.g., \cite{Bila, BKPa, BiZi, CCE, CPZ, GMS, OgWa, Naib, SoWi, Winb}). 

When $N \geq 3$ the system (\ref{eq1.1}) has the singular steady state explicitly given by  
\begin{equation}
	U_{\infty}(|x|) = \frac{2N-4}{|x|^2},
	\label{eq1.2}
\end{equation}
which is known as the Chandrasekhar solution \cite{Cha}.
It was shown by Biler, Karch and Pilarczyk \cite{BKPa} 
that if radial nonnegative initial function $u_0(|x|)$ in 
$M^{N/2}(\R^N) \cap M^p(\R^N)$ with $p \in (N/2, N)$ satisfies 
\begin{equation}
	\sup_{R > 0}R^{2-N}\int^R_0 r^{N-1}u_0(r)dr < 2, 
	\label{eq1.3}
\end{equation}
then the corresponding solution $u$ of (\ref{eq1.1}) is global and bounded, where 
$M^p(\R^N)$ is the homogeneous Morrey space (see also \cite{Bilc}). 
Since $U_{\infty}(r)$, defined by (\ref{eq1.2}), satisfies  
$$
	R^{2-N}\int^R_0 r^{N-1}U_{\infty}(r)dr = 2 \quad \mbox{for each } \ R > 0,
$$
the assumption (\ref{eq1.3}) means that the initial datum $u_0$ is strictly below $U_{\infty}$ 
in the sense of radial concentrations.
Analogue results are also obtained for the problem in bounded domains by Winkler \cite{Wina}. 
Furthermore, in the case $N \geq 10$, it was shown by \cite[Theorem 1.4 (i)]{Naib} that, 
for each $\ell > 2$, there exists a radial nonnegative initial function $u_0$ such that 
$$
	\sup_{R > 0}R^{2-N}\int^R_0 r^{N-1}u_0(r)dr = \ell
$$
and the corresponding solution $u$ of (\ref{eq1.1}) blows up in finite time.
Thus the assumption (\ref{eq1.3}) is optimal for the existence of global solutions when $N \geq 10$. 
On the other hand, when $3 \leq N \leq 9$,  
there exists a radial nonnegative initial function $u_0$ such that 
$$
	\sup_{R > 0}R^{2-N}\int^R_0 r^{N-1}u_0(r)dr > 2
$$
and the corresponding solution $u$ of (\ref{eq1.1}) is global and bounded 
(see \cite[Theorem 1.4 (ii)]{Naib}).   
It was shown that the decay behavior of $U_{\infty}$ is critical with 
respect to the occurrence of finite time blow-up (see \cite{Winb}).  

In the case $u_0(x) = c|x|^{-2}$ with a constant $c > 0$, 
the system (\ref{eq1.1}) is invariant under the similarity transformation 
$$
	u_{\lam}(x, t) = \lam^2 u(\lam x, \lam^2 t) \quad \mbox{and} \quad 
	v_{\lam}(x, t) = v(\lam x, \lam^2 t) \quad \mbox{for} \ \lam > 0.
$$
We say that a solution $(u, v)$ is self-similar if $(u, v)$ satisfies 
$u_{\lam}(x, t) = u(x, t)$ and $v_{\lam}(x, t) = v(x, t)$ for all 
$x \in \R^N$, $t > 0$ and $\lam > 0$. 
It is worth mentioning that Biler, Karch and Wakui \cite{BKW} constructed 
radial self-similar solutions with initial data  
$u_0(x) = \vep U_{\infty}(|x|)$ for $0 < \vep < 1$. 

In this paper, we always assume in (\ref{eq1.1}) that 
$u_0$ is radially symmetric, i.e., $u_0 = u_0(|x|)$ with $x \in \R^N$, 
and  $u_0(r)$ satisfies 
\begin{equation}
	u_0 \in C[0, \infty) \quad \mbox{and} \quad 
	0 < u_0(r) < U_{\infty}(r) \quad \mbox{for all} \ r > 0.
	\label{eq1.4}
\end{equation}
By \cite[Appendix A]{Naib} (see also \cite[Proposition 1.1]{Winb}), 
if $u_0(|x|)$ belongs to $C(\R^N)\cap L^{\infty}(\R^N)$ and is nonnegative, 
then there exists $T_0 = T_0(\|u_0\|_{L^{\infty}[0, \infty)}) > 0$ such that 
(\ref{eq1.1}) has a unique positive radial solution 
$$
	u \in C^{2, 1}(\R^N\times (0, T_0]) \cap C(\R^N\times [0, T_0]) 
	\cap L^{\infty}(0, T_0; L^{\infty}(\R^N)).
$$
Furthermore, it was shown by Winkler \cite{Winb} 
(see also \cite{Sena}) that the solution exists globally if 
$u_0$ satisfies (\ref{eq1.4}). 

\begin{theoremA}
{\rm (\cite[Theorem 1.2]{Winb})} 
Assume that $u_0$ is radially symmetric and satisfies $(\ref{eq1.4})$.
Then the solution $(u, v)$ of $(1.1)$ is global and satisfies
\begin{equation}
	\frac{1}{|B_r(0)|}\int_{B_r(0)}u(x, t)dx \leq \frac{2N}{r^2} 
	\quad \mbox{for all $r > 0$ and $t > 0$}.
	\label{eq1.5}
\end{equation}
\end{theoremA}

The condition (\ref{eq1.5}) does not guarantee 
whether the blow-up in infinite time will occur or not.
In \cite{Sena} one of the author showed the existence of infinite time blow-up solutions 
when $N \geq 11$. 
Furthermore, various behavior of global solutions, as well as global unbounded solutions, 
were also constructed in \cite{Senb}. 
Recently, Colasuonno and Winkler \cite{CoWi} investigated 
the stability and attractiveness of the singular steady state $U_{\infty}$ for all dimensions $N \geq 3$, 
and showed the existence of infinite time blow-up solutions to (\ref{eq1.1}) when $N \geq 10$.  
For $N \geq 10$, define 
\begin{equation}
	\lam_0 = \frac{N-2-\sqrt{(N-2)(N-10)}}{2}.
	\label{eq1.6}
\end{equation}

\begin{theoremB}
{\rm (\cite[Theorems 1.2 and 1.3]{CoWi})} 
Assume that $u_0$ is radially symmetric and satisfies $(\ref{eq1.4})$.
\begin{itemize}
\item[{\rm (i)}] 
In the case $3 \leq N \leq 9$, there exists a constant $C > 0$ such that 
the solution $(u, v)$ of $(1.1)$ satisfies $\|u(\cdot, t)\|_{L^{\infty}(\R^N)} \leq C$ 
for all $t \geq t_*$ 
with some $t_* = t_*(u_0) > 0$. 

\item[{\rm (ii)}] 
In the case $N \geq 10$, assume in addition that $u_0$ satisfies   
$$
	u_0(r) \geq \frac{2N-4}{r^2} -\frac{C}{r^{2+\theta}} 
	\quad \mbox{for all} \ r \geq 1
$$
with some constants $C > 0$ and $\theta > \lam_0$, 
where $\lam_0$ is defined by $(\ref{eq1.6})$. 
Then the solution $(u, v)$ of $(\ref{eq1.1})$ blows up
infinite time in the sense that 
$\|u(\cdot, t)\|_{L^{\infty}(\R^N)} \to \infty$ as $t \to \infty$. 
Moreover, $u(\cdot, t) \to U_{\infty}(|\cdot|)$ in $C_{\rm loc}(\R^N\setminus\{0\})$
as $t \to \infty$.
\end{itemize}
\end{theoremB}

In this paper we consider the optimal criteria for the existence of 
infinite time blow-up solutions when $N \geq 10$.
We first consider the asymptotic expansion of radial solutions to 
the stationary problem
\begin{equation}
	\left\{
	\begin{array}{ll}
	0 = \nabla\cdot (\nabla U - U\nabla V) 
	\quad \mbox{in} \ \R^N,
	\\[1ex]
	0 = \Delta V + U \quad  \mbox{in} \ \R^N. 
	\end{array}
	\right.
	\label{eq1.7}
\end{equation}
For $\alpha > 0$, we denote by $(U_{\alpha}(r), V_{\alpha}(r))$, with $r = |x|$, 
a radially symmetric solution of (\ref{eq1.7}) satisfying 
\begin{equation}
	U_{\alpha}(0) = \alpha, \quad V_{\alpha}(0) = 0
	\quad \mbox{and} \quad U_{\alpha}'(0) = V_{\alpha}'(0) = 0.
	\label{eq1.8}
\end{equation}
We see that the exponent $\lam_0$, defined by (\ref{eq1.6}), appears in the 
asymptotic expansion of $U_{\alpha}(r)$ as $r \to \infty$.

\begin{proposition}\label{prp1.1}
For $\alpha > 0$, the solution $U_{\alpha}$ satisfies
\begin{equation}
	U_{\alpha}(r) = U_{\infty}(r) -C_N(\alpha)r^{-2-\lam_0} + o(r^{-2-\lam_0}) 
	\quad \mbox{as} \ r \to \infty 
	\label{eq1.9}
\end{equation}
if $N \geq 11$, and
\begin{equation}
	U_{\alpha}(r) = U_{\infty}(r) -C_N(\alpha)r^{-2-\lam_0}\log r
	+ o(r^{-2-\lam_0}\log r)
	\quad \mbox{as} \ r \to \infty
	\label{eq1.10}
\end{equation}
if $N = 10$, 
where $C_N(\alpha)$ is a positive constant.
In the both $(\ref{eq1.9})$ and $(\ref{eq1.10})$, the constant $C_N(\alpha)$ satisfies 
\begin{equation}
	C_N(\alpha) = \alpha^{-\frac{\lam_0}{2}}C_N(1)
	\quad \mbox{for} \ \alpha > 0.
	\label{eq1.11}
\end{equation}
\end{proposition}

We next consider the optimal criteria for the existence of infinite time blow-up solutions. 
The following result says that the exponent $\lam_0$ 
in Theorem B is optimal when $N \geq 11$.

\begin{theorem}\label{thm1.2}
Let $N \geq 11$. 
Assume that $u_0(r)$ is radially symmetric and satisfies $(\ref{eq1.4})$.
\begin{itemize}
\item[{\rm (i)}]
If 
\begin{equation}
	\lim_{r \to \infty}r^{2+\lam_0}\left(U_{\infty}(r)- u_0(r)\right)=  0,
	\label{eq1.12}
\end{equation}
then the solution $(u, v)$ of $(\ref{eq1.1})$ satisfies 
$\|u(\cdot, t)\|_{L^{\infty}(\R^N)} \to \infty$ as $t \to \infty$.
\item[{\rm (ii)}] 
If 
\begin{equation}
	\liminf_{r \to \infty}r^{2+\lam_0}\left(U_{\infty}(r)- u_0(r)\right) >  0,
	\label{eq1.13}
\end{equation}
then the solution $(u, v)$ of $(\ref{eq1.1})$ satisfies 
$\sup_{t > 0}\|u(\cdot, t)\|_{L^{\infty}(\R^N)} < \infty$.
\end{itemize}
\end{theorem}

In the case $N = 10$, we obtain the following criteria which involves a logarithmic correction.

\begin{theorem}\label{thm1.3}
Let $N = 10$. 
Assume that $u_0(r)$ is radially symmetric and satisfies $(\ref{eq1.4})$.
\begin{itemize}
\item[{\rm (i)}] If 
\begin{equation}
	\lim_{r \to \infty}\frac{r^{2+\lam_0}}{\log r}\left(U_{\infty}(r)- u_0(r)\right) =  0,
	\label{eq1.14}
\end{equation}
then the solution $(u, v)$ of $(\ref{eq1.1})$ satisfies 
$\|u(\cdot, t)\|_{L^{\infty}(\R^N)} \to \infty$ as $t \to \infty$.

\item[{\rm (ii)}] If 
\begin{equation}
	\liminf_{r \to \infty}\frac{r^{2+\lam_0}}{\log r}\left(U_{\infty}(r)- u_0(r)\right) >  0,
	\label{eq1.15}
\end{equation}
then the solution $(u, v)$ of $(\ref{eq1.1})$ satisfies 
$\sup_{t > 0}\|u(\cdot, t)\|_{L^{\infty}(\R^N)} < \infty$.
\end{itemize}
\end{theorem}

It should be mentioned that criteria obtained in Theorems \ref{thm1.2} and \ref{thm1.3} 
have a close relation with the asymptotic expansion of $U_{\alpha}$ obtained 
in Proposition \ref{prp1.1}.
By (\ref{eq1.11}) we have $C_N(\alpha) \to 0$ as $\alpha \to \infty$. 
Then Theorems \ref{thm1.2} and \ref{thm1.3} say that 
the solution $(u, v)$ blows up in infinite time if  
$u_0(r) > U_{\alpha}(r)$ for sufficiently large $r > 0$ with any $\alpha > 0$,  
while the solution $(u, v)$ is remain bounded if   
$u_0(r) \leq U_{\alpha}(r)$ for sufficiently large $r > 0$ with some $\alpha > 0$. 
Note that the relation (\ref{eq1.11}) follows from the scaling property of (\ref{eq1.1}).


The idea of the proof of the theorems is to reduce the system (\ref{eq1.1}) 
into a scalar equation involving the averaged mass of $u$ (see, e.g., \cite{BCK, GuPe, SoWi, Naib}).
Let $(u, v)$ be a radially symmetric solution of (\ref{eq1.1}) for $0 \leq t < T$.
Define 
\begin{equation}
	w(r, t) = \frac{1}{r^N}\int^r_0 s^{N-1}u(s, t) ds \quad \mbox{for} \ r > 0, \ 0 \leq t < T.
	\label{eq1.16}
\end{equation}
Then $w$ satisfies
\begin{equation}
	w_t = w_{rr} + \frac{N+1}{r}w_r + rw_rw + Nw^2 \quad 
	\mbox{for} \ r > 0, \ 0 \leq t < T,
	\label{eq1.17}
\end{equation}
and $w(r, 0) = w_0(r)$ for $r \geq 0$, where 
\begin{equation}
	w_0(r) =  \frac{1}{r^N}\int^r_0 s^{N-1}u_0(s) ds \quad \mbox{for} \ r > 0
	\quad \mbox{and} \quad  w_0(0) = \frac{u_0(0)}{N}.
	\label{eq1.18}
\end{equation}
It can be seen from (\ref{eq1.17}) that $w = w(r, t)$ is a radially symmetric solution of 
the equation 
$$
	w_t = \Delta w + (x\cdot\nabla w)w + Nw^2 \quad \mbox{for} \ 
	x \in \R^{N+2}, \ 0 < t < T,
$$
where $\Delta$ and $\nabla$, respectively, are Laplacian and the spatial gradient 
in $N+2$ space variables.
We will study the Cauchy problem for the transformed scalar equation 
\begin{equation}
	\left\{
	\begin{array}{ll}
	w_t = \Delta w + (x\cdot \nabla w)w + Nw^2, \quad &
	x \in \R^{N+2}, \ t > 0, 
	\\[2ex]
	w(x, 0) = w_0(|x|) \geq 0, \quad & x \in \R^{N+2},
	\end{array}
	\right.
	\label{eq1.19}
\end{equation}
where the comparison principle holds. (See Proposition 2.1 below.) 
We also investigate the properties of solutions to the stationary equation
\begin{equation}
	\Delta w + (x\cdot \nabla w)w + Nw^2 = 0 \quad \mbox{in} \ \R^{N+2}.
	\label{eq1.20}
\end{equation}
We will obtain the criteria in Theorems \ref{thm1.2} and \ref{thm1.3} by making use of 
super- and subsolutions of (\ref{eq1.20}) 
together with the comparison argument for solutions of (\ref{eq1.19}).  
It should be mentioned that regular solutions of (\ref{eq1.20}) 
and the stationary solutions of (\ref{eq1.7}) 
can be represented in terms of solutions to the Liouville equation 
\begin{equation}
	\Delta\phi + e^{\phi} = 0 \quad \mbox{in} \ \R^N.
	\label{eq1.21}
\end{equation}
We will investigate the asymptotic expansion of regular radial solutions of (\ref{eq1.21}) 
to show Proposition \ref{prp1.1}. 
We also construct super- and subsolutions to (\ref{eq1.20}) by using some properties of 
regular solutions to (\ref{eq1.21}).

The paper is organized as follows. 
In Section 2, we recall preliminary results for the Cauchy problem (\ref{eq1.19}) and 
its stationary problem (\ref{eq1.20}). 
In Section 3 we investigate the property of solutions to (\ref{eq1.21}) and 
state Proposition \ref{prp3.1} and Lemmas \ref{lem3.3} and \ref{lem3.4}. 
We also give the proof of Proposition \ref{prp1.1} in Section 3. 
In Sections 4 and 5, we give the proof of Theorems \ref{thm1.2} and \ref{thm1.3} by 
employing the results obtained in Section 3.
In Section 6 we give the proof of Proposition \ref{prp3.1}, 
and we prove Lemmas \ref{lem3.3} and \ref{lem3.4} in Section 7.


\section{Preliminaries}

We first recall the following comparison principle \cite[Proposition 2.1]{Naib}.
We set $Q_{T_0} = \R^{N+2} \times (0, T_0]$ for $T_0 > 0$. 

\begin{proposition}\label{prp2.1}
Assume that $w_1, w_2 \in C^{2, 1}(Q_{T_0})\cap C(\R^{N+2}\times [0, T_0])$ satisfy 
$$
	\partial_t w_1 - \Delta w_1 -(x\cdot \nabla w_1)w_1 -Nw_1^2 
	\leq 
	\partial_t w_2 - \Delta w_2 -(x\cdot \nabla w_2)w_2 -Nw_2^2 
$$
in $Q_{T_0}$, 
where $\Delta$ and $\nabla$, respectively, are 
Laplacian and the spatial gradient in $\R^{N+2}$.
Assume that $\sup_{Q_{T_0}}|w_i| < \infty$ for $i = 1, 2$,  
and that either
$x \cdot \nabla w_1$ or $x \cdot \nabla w_2$ is bounded above in $Q_{T_0}$. 
If $w_1(x, 0) \leq w_2(x, 0)$ for $x \in \R^{N+2}$ then 
$w_1 \leq w_2$ in $Q_{T_0}$.
\end{proposition}

Let us consider the Cauchy problem (\ref{eq1.19}) and 
the corresponding stationary equation (\ref{eq1.20}). 
We say that $w \in C^2(\R^{N+2})$ is a subsolution of (\ref{eq1.20}) if 
$w$ satisfies
$$
	\Delta w + (x\cdot \nabla w)w + Nw^2 \geq 0 \quad \mbox{for} \ x \in \R^{N+2}.
$$
By \cite[Proposition 2.3]{Naib} we obtain the following.

\begin{proposition}\label{prp2.2}
Let $w \in C^{2,1}(\R^{N+2}\times (0, \infty))\cap C(\R^{N+2}\times [0, \infty))$ 
be a global in time solution of $(\ref{eq1.19})$. 
Assume that $x\cdot \nabla w$ and $w$ 
are bounded above for all $x \in \R^{N+2}$ and $t > 0$. 
If $w_0 \in C^2(\R^{N+2}) \cap L^{\infty}(\R^{N+2})$ is 
a subsolution of $(\ref{eq1.20})$, 
then $w(x, t)$ is nondecreasing in $t > 0$ for each fixed $x \in \R^{N+2}$.
\end{proposition}

We prepare the following convergence result by \cite[Proposition 3.1]{Naib}.

\begin{proposition}\label{prp2.3}
Let $w$ be a global in time solution of $(\ref{eq1.19})$.
Assume that $w(x, t)$ is nondecreasing in $t \geq 0$ for 
each fixed $x \in \R^{N+2}$, and put $W(x) = \lim_{t \to \infty}w(x, t)$ 
for each $x \in \R^{N+2}$. 
Then $W \in L^1_{\rm loc}(\R^{N+2})\cap L^2_{\rm loc}(\R^{N+2})$ and 
$W$ solves 
\begin{equation}
	\Delta W + (x\cdot \nabla W)W + NW^2 = 0 
	\label{eq2.3}
\end{equation}
for $x \in \R^{N+2}$ in the sense of distribution. 
Furthermore, if $W \in L^{\infty}(\R^{N+2})$ then $W \in C^2(\R^{N+2})$ 
is a classical solution of $(\ref{eq2.3})$ in $\R^{N+2}$.
\end{proposition}

Let us consider the following ordinary differential equation
\begin{equation}
	W'' + \frac{N+1}{r}W' + rW'W + NW^2 = 0 
	\quad \mbox{for}\  r > 0.
	\label{eq2.4}
\end{equation}
We denote by $\phi(r; \log\alpha)$ a unique solution of the initial value problem
\begin{equation}
	\left\{
	\begin{array}{c}
	(r^{N-1}\phi')' + r^{N-1}e^{\phi} = 0 \quad \mbox{for} \ r > 0,
	\\[1ex]
	\phi(0) = \log\alpha, \quad \phi'(0) = 0.
	\end{array}
	\right.
	\label{eq2.5}
\end{equation}
It was shown by \cite[Proposition 4.1]{Naib} that   
regular solutions of (\ref{eq2.4}) can be 
represented by the solution $\phi(r, \log \alpha)$ of (\ref{eq2.5}). 

\begin{proposition}\label{prp2.4}
A function $W \in C^2(0, \infty)\cap C^1[0, \infty)$ is a solution of $(\ref{eq2.4})$ 
satisfying $W(0) = \alpha > 0$ if and only if $W(r)$ is given by 
$$
	W(r) = \frac{1}{r^N}\int^r_0 s^{N-1}e^{\phi(s, \log (N\alpha))}ds 
	\quad \mbox{for} \ r > 0
$$
and $W(0) = \alpha$, 
where $\phi(r, \log \alpha)$ is the solution of $(\ref{eq2.5})$.
Furthermore, the solution $W$ satisfies 
$W \in C^2[0, \infty)$, $W'(0) = 0$ and $W'(r) \leq 0$ for $r \geq 0$. 
\end{proposition}

Proposition \ref{prp2.4} implies that $x\cdot \nabla W(|x|) = rW_r(r) \leq 0$ 
for $r = |x|$ with $x \in \R^{N+2}$. 
Applying Proposition \ref{prp2.1} with solutions of (\ref{eq2.3}) and (\ref{eq1.19}), 
we obtain the following result.

\begin{corollary}\label{cor2.5}
Let $W \in C^2(\R^{N+2})$ be a positive radial solution of $(\ref{eq2.3})$ in $\R^{N+2}$, 
and let $w$ be a solution of $(\ref{eq1.19})$. 
If $w_0(x) \leq W(|x|)$ in $\R^{N+2}$, then $w(x, t)$ exists globally in time and 
satisfies $w(x, t) \leq W(|x|)$ for all $(x, t) \in (\R^{N+2}\times [0, \infty))$.
\end{corollary}

We recall some properties of $w$ defined by (\ref{eq1.16}).
(See \cite[Proposition 3.2]{SoWi} and \cite[Lemma 5.1]{Naib}.)  

\begin{lemma}\label{lem2.6}
Let $(u, v)$ be a solution of $(\ref{eq1.1})$ for $0 \leq t < T \leq \infty$, 
and define $w(r, t)$ by $(\ref{eq1.16})$. 
Then $w(0, t) = u(0, t)/N$ for $t \geq 0$. 
Assume in addition that $u_0(r)$ is nonincreasing.  
Then $w_r(r, t) \leq 0$ for $r \geq 0$ and $0 < t < T$.
\end{lemma}

Using parabolic a priori estimates for  the one-dimensional linear parabolic equations 
(see, for instance, \cite[Theorem 11.1 in Chapter III]{LSU}),
 we obtain the following proposition.

\begin{proposition}\label{prp2.7}
Let $(u, v)$ be a radial global solution of $(\ref{eq1.1})$ with $u_0$ satisfying $(\ref{eq1.4})$.  
Assume that $w(r, t)$, defined by $(\ref{eq1.16})$, 
is uniformly bounded for $(r, t) \in [0, \infty)\times [0, \infty)$. 
Then the solution $u(r, t)$ is also uniformly bounded for $(r, t) \in [0, \infty)\times [0, \infty)$. 
\end{proposition}

\begin{proof}
Since $u_0(|x|)$ belongs to $C(\R^N)\cap L^{\infty}(\R^N)$, 
the solution $u$ of (1.1) exists and is bounded on $\R^N\times [0, a]$ 
with some $a > 0$ 
by \cite[Appendix A]{Naib} (see also \cite[Proposition 1.1]{Winb}). 
Let $t \geq a$ and $0 < r \leq 1$ be fixed, and define
$$
	z(\rho, \tau) = w((\rho + 1)r, t + r^2\tau) \quad 
	\mbox{for} \ (\rho, \tau) \in (-1, 1)\times (-a, a).
$$
A straightforward calculation yields that 
$$
	z_{\tau} = z_{\rho\rho}+ \frac{N+1}{\rho +1}z_{\rho} 
	+ r^2z\left((\rho + 1)z_{\rho} + Nz\right) 
	\quad \mbox{for} \ (\rho, \tau) \in (-1, 1)\times (-a, a).
$$
Define $M_w = \sup_{t \geq 0}\|w(\cdot, t)\|_{L^{\infty}[0, \infty)}$. 
Then $Z(\rho, \tau) = z(\rho, \tau)/M_w$ satisfies $|Z(\rho, \tau)| \leq 1$ and 
\begin{equation}
	Z_{\tau} = Z_{\rho\rho} + a(\rho, \tau; r)Z_{\rho} + b(\rho, \tau; r)Z
	\quad \mbox{for} \ (\rho, \tau) \in (-1, 1)\times (-a, a),
	\label{eq2.6}
\end{equation}
where 
$$
	a(\rho, \tau; r) = \frac{N+1}{\rho+1} + r^2(\rho+1)z(\rho, \tau)
	\quad \mbox{and} \quad 
	b(\rho, \tau; r) = Nr^2z(\rho, \tau).
$$
Since $|z(\rho, \tau)| \leq M_w$ and $r^2 \leq 1$, we have 
$$
	|a(\rho, \tau; r)| \leq 2(N+1) + \frac{3}{2}M_w 
	\quad \mbox{and} \quad 
	|b(\rho, \tau; r)| \leq NM_w 
$$
for $(\rho, \tau) \in [-1/2, 1/2]\times [-a/2, a/2]$.
By \cite[Theorem 11.1 in Chapter III]{LSU}, we deduce that 
\begin{equation}
	|Z_{\rho}(0, 0)| \leq 
	\sup\{|Z_{\rho}(\rho, \tau)|: (\rho, \tau) \in [-1/4, 1/4]\times [-a/4, a/4]\} 
	\leq C
	\label{eq2.7}
\end{equation}
with a constant $C = C(N, M_w) > 0$ which is independent of $0 \leq r \leq 1$ and $t \geq a$.
Since $|Z_{\rho}(0, 0)| = |z_{\rho}(0, 0)|/M_w = |rw_r(r, t)|/M_w$, 
we obtain $|rw_r(r, t)| \leq C M_w$ for $0 \leq r \leq 1$ and $t \geq a$. 

Next, let $t \geq a$ and $r \geq 1$ be fixed, and define $z(\rho, \tau)$ and $M_w(r, t)$ 
respectively by 
$$
	z(\rho, \tau) = w(r + \rho, t +\tau) \quad 
	\mbox{for} \ (\rho, \tau) \in (-1, 1)\times (-a, a)
$$
and 
$$
	M_w(r, t) = \sup\{w(\tilde{r}, \tilde{t}): r-1 \leq \tilde{r} \leq r+1, t-a \leq \tilde{t} \leq t+a\}.
$$ 
Then $Z(\rho, \tau) = z(\rho, \tau)/M_w(r, t)$ satisfies $|Z(\rho, \tau)| \leq 1$ and 
(\ref{eq2.6}) with 
$$
	a(\rho, \tau; r) = \frac{N+1}{\rho+r} + (\rho+r)z(\rho, \tau)
	\quad \mbox{and} \quad 
	b(\rho, \tau; r) = Nz(\rho, \tau).
$$
Since $u_0$ in (\ref{eq1.1}) satisfies (\ref{eq1.4}), 
we obtain $w(r, t) \leq 2/r^2$ for $r > 0$ and $t > 0$ by Theorem A. 
This implies that 
$$
	|z(\rho, \tau)| \leq \frac{2}{(r+ \rho)^2} \quad 
	\mbox{for} \ (\rho, \tau) \in (-1, 1)\times (-a, a).
$$
From $r \geq 1$, it follows that 
$$
	|(\rho+r)z(\rho, \tau)| \leq \frac{2}{(r+ \rho)} \leq 4 \quad 
	\mbox{for} \ (\rho, \tau) \in [-1/2, 1/2]\times [-a/2, a/2].
$$
Then we obtain 
$|a(\rho, \tau; r)| \leq 2(N+1) + 4$ and $|b(\rho, \tau; r)| \leq NM_w$ 
for $(\rho, \tau) \in [-1/2, 1/2]\times [-a/2, a/2]$. 
By \cite{LSU}, we obtain (\ref{eq2.7}) 
with a constant $C = C(N, M_w) > 0$ which is independent of $r \geq 1$ and $t \geq a$. 
Since $Z_{\rho}(0, 0) = z_{\rho}(0, 0)/M_w(r, t) = w_r(r, t)/M_w(r, t)$, we obtain 
$$
	|rw_r(r, t)| \leq C rM_w(r, t) 
	\quad \mbox{for} \ r \geq 1,  \ t \geq a.
$$
Recall that $w$ satisfies $w(r, t) \leq M_w$ and $w(r, t) \leq 2/r^2$ for $r > 0$ and $t > 0$. 
Then 
$rM_w(r, t) \leq 2M_w$ for $1 \leq r \leq 2$, $t \geq a$ and  
$$
	rM_w(r, t) \leq \frac{2r}{(r-1)^2} \quad \mbox{for} \ r \geq 2, \ t \geq a.
$$
Then we obtain $|rw_r(r, t)| \leq C$ for $r \geq 1$ and $t \geq a$ with some constant $C > 0$.

Differentiating (\ref{eq1.16}) by $r > 0$, we have 
$$
	Nw(r, t) + rw_r(r, t) = u(r, t) \quad \mbox{for} \ r > 0, \ t > 0.
$$
Then we obtain $u(r, t) \leq NM_w +C$ for $r > 0$ and $t \geq a$.
Thus we conclude that 
$u(r, t)$ is uniformly bounded for $(r, t) \in [0, \infty) \times [0, \infty)$.
\end{proof}


\section{Asymptotic expansion of solutions to the Liouville equation}

Note that the equation in (\ref{eq2.5}) is invariant under the scaling transformation 
$$
	\phi_{\lam}(r) = \phi(\sqrt{\lam} r) + \log \lam \quad \mbox{for} \ \lam > 0.
$$
By the scaling property, the solution $\phi(r, \log \alpha)$ of (\ref{eq2.5}) satisfies
\begin{equation}
	\phi(r, \log \alpha) = \phi(\sqrt{\alpha}r, 0) + \log \alpha
	\quad \mbox{for} \ \alpha > 0.
	\label{eq3.1}
\end{equation}
When $N \geq 3$, the equation in (\ref{eq2.5}) has a singular solution 
\begin{equation}
	\phi_{\infty}(r) = -2\log r + \log(2N-4) \quad \mbox{for} \  r > 0.
	\label{eq3.2}
\end{equation}
By the phase plane argument(see, e.g., the proof of Theorem 1 in \cite{Tel}), 
for each $\alpha > 0$ we have  
\begin{equation}
	\phi(r, \log \alpha) \to \phi_{\infty}(r) \quad \mbox{as} \ r \to \infty.
	\label{eq3.3}
\end{equation}
We also obtain 
\begin{equation}
	\lim_{\alpha \to \infty}\phi(r, \log \alpha) = \phi_{\infty}(r)
	\quad \mbox{uniformly on $[r_0, \infty)$} 
	\label{eq3.4}
\end{equation}
for any $r_0 > 0$. 
In fact, from (\ref{eq3.1}) and (\ref{eq3.2}), we have 
$$
	\phi(r, \log \alpha) - \phi_{\infty}(r) = 
	\phi(\sqrt{\alpha}r, 0) + \log \alpha -\phi_{\infty}(r) = 
	\phi(\sqrt{\alpha}r, 0) - \phi_{\infty}(\sqrt{\alpha}r).
$$ 
By (\ref{eq3.3}) we obtain (\ref{eq3.4}). 

In the case $N \geq 10$ it was shown by \cite[Theorem 1.1]{Tel} that, if $0 < \alpha < \beta$, then 
\begin{equation}
	\phi(r, \log \alpha) < \phi(r, \log \beta) < \phi_{\infty}(r)  \quad \mbox{for} \ r > 0.
	\label{eq3.5}
\end{equation}
Furthermore, we obtain the following. 

\begin{proposition}\label{prp3.1}
The solution $\phi(r, \log \alpha)$ of $(\ref{eq2.5})$ satisfies
\begin{equation}
	\phi(r, \log \alpha) = \phi_{\infty}(r) -\ell_N(\alpha)r^{-\lam_0} + o(r^{-\lam_0}) 
	\quad \mbox{as} \ r \to \infty 
	\label{eq3.6}
\end{equation}
if $N \geq 11$, and
\begin{equation}
	\phi(r, \log \alpha) = \phi_{\infty}(r) -\ell_N(\alpha)r^{-\lam_0}\log r
	+ o(r^{-\lam_0}\log r)
	\quad \mbox{as} \ r \to \infty
	\label{eq3.7}
\end{equation}
if $N = 10$, where $\ell_N(\alpha)$ is a positive constant.
In the both $(\ref{eq3.6})$ and $(\ref{eq3.7})$, 
the constant $\ell_N(\alpha)$ satisfies 
\begin{equation}
	\ell_N(\alpha) = \alpha^{-\frac{\lam_0}{2}}\ell_N(1) \quad \mbox{for} \ \alpha > 0. 
	\label{eq3.8}
\end{equation}
\end{proposition}

\begin{remark}
By the usual argument as in \cite{JoLu, GNW} 
we obtain standard asymptotic expansions (see, e.g., \cite[Lemma 2.1]{Tel}). 
In Proposition \ref{prp3.1} we clarify the dependence of the parameter $\ell_N(\alpha)$ 
with respect to $\alpha > 0$,
which plays a crucial role in our proof of Theorems \ref{thm1.2} and \ref{thm1.3}.
\end{remark}

We will give the proof of Proposition \ref{prp3.1} in Section 6 below.
By Proposition \ref{prp3.1} we obtain the following corollary, which will be used in the proof 
of Proposition \ref{prp1.1} and Theorems \ref{thm1.2} and \ref{thm1.3}. 

\begin{corollary}\label{cor3.2}
For $\alpha > 0$, let $\phi(r, \log \alpha)$ be the solution of $(\ref{eq2.5})$. 
Then  
\begin{equation}
	e^{\phi(r, \log \alpha)} = 
	\frac{2N-4}{r^2} -(2N-4)\ell_N(\alpha)r^{-2-\lam_0} + o(r^{-2-\lam_0})
	\label{eq3.9}
\end{equation}
as $r \to \infty$ if $N \geq 11$, and 
\begin{equation}
	e^{\phi(r, \log \alpha)} = 
	\frac{2N-4}{r^2} -(2N-4)\ell_N(\alpha)r^{-2-\lam_0}\log r 
	+ o(r^{-2-\lam_0}\log r) 
	\label{eq3.10}
\end{equation}
as $r \to \infty$ if $N = 10$,
where $\ell_N(\alpha)$ is the constant in Proposition $\ref{prp3.1}$. 
\end{corollary}

\begin{proof}
Let us consider the case $N \geq 11$. 
By (\ref{eq3.6}) we have
$$
	\dsp e^{\phi(r, \log \alpha)} = 
	\frac{2N-4}{r^2}e^{-\ell_N(\alpha)r^{-\lam_0} + o(r^{-\lam_0})}
	\quad \mbox{as} \ r \to \infty.
$$
Since $e^x = 1 + x + o(x^2)$ as $x \to 0$, we have
$$
	\begin{array}{rcl}
	\dsp e^{\phi(r, \log \alpha)} & = & \dsp \frac{2N-4}{r^2}
	\left(1-\ell_N(\alpha)r^{-\lam_0} + o(r^{-\lam_0})\right)
	\\[2ex] 
	& = & \dsp \frac{2N-4}{r^2} -(2N-4)\ell_N(\alpha)r^{-2-\lam_0} + o(r^{-2-\lam_0})
	\end{array}
$$
as $r \to \infty$, and hence (\ref{eq3.9}) holds. 
In the case $N = 10$, by the same argument we obtain (\ref{eq3.10}). 
\end{proof}

Proposition \ref{prp1.1} follows immediately from Corollary \ref{cor3.2}.

\begin{proof}[Proof of Proposition \ref{prp1.1}] 
We see that the radially symmetric solution $(U_{\alpha}, V_{\alpha})$ 
of (\ref{eq1.7}) satisfying (\ref{eq1.8}) is given by 
\begin{equation}
	U_{\alpha}(r) = e^{\phi(r; \log \alpha)} \quad 
	\mbox{and} \quad 
	V_{\alpha}(r) = \phi(r; \log \alpha) -\log \alpha \quad \mbox{for} \ r \geq 0. 
	\label{eq3.11}
\end{equation}
In fact, by a direct calculation, 
$(U_{\alpha}(|x|), V_{\alpha}(|x|))$ with $x \in \R^N$ solves 
(\ref{eq1.7}) and satisfies (\ref{eq1.8}). 
By the uniqueness of the initial value problem, the solution $(U_{\alpha}, V_{\alpha})$ 
is given by (\ref{eq3.11}).
By Corollary \ref{cor3.2}, we obtain (\ref{eq1.9}) and (\ref{eq1.10}) 
with $C_N(\alpha) = (2N-4)\ell_N(\alpha)$. 
From (\ref{eq3.8}), we see that $C_N(\alpha)$ satisfies (\ref{eq1.11}). 
\end{proof}

We assume that $\phi_0 \in C[0, \infty)$ satisfies  
\begin{equation}
	\phi_0(r) > 0 \quad \mbox{for} \ r \geq 0 \quad \mbox{and} \quad 
	\phi_0(r) < \phi_{\infty}(r) \quad \mbox{for} \ r > 0.
	\label{eq3.12}
\end{equation}
In the proof of Theorems \ref{thm1.2} and \ref{thm1.3}, 
we need the following lemmas.

\begin{lemma}\label{lem3.3}
Let $N \geq 10$. 
Assume that $\phi_0$ satisfies $(\ref{eq3.12})$ and 
there exist $\alpha_1 > 0$ and $r_1 > 0$ such that 
\begin{equation}
	\phi_0(r) \leq \phi(r, \log \alpha_1) \quad \mbox{for} \ r \geq r_1.
	\label{eq3.13}
\end{equation}
Then there exists $\hat{\alpha} \geq \alpha_1$ such that  
$\phi_0(r) \leq \phi(r, \log \hat{\alpha})$ for all $r \geq 0$.
\end{lemma}

\begin{lemma}\label{lem3.4}
Let $N \geq 10$. 
Assume that $\phi_0(r)$ is nonincreasing and  satisfies $(\ref{eq3.12})$.
Assume, in addition, that there exist $\alpha > 0$ and $r_1 > 0$ such that 
\begin{equation}
	\phi_0(r) \geq \phi(r, \log \alpha) \quad \mbox{for} \ r \geq r_1
	\quad \mbox{and} \quad 
	\sup_{r \geq 0}(\phi(r, \log \alpha) - \phi_0(r)) > 0.
	\label{eq3.14}
\end{equation}
Then there exists $\underline{\phi}_0 \in C^1[0, \infty)$ satisfying the following. 
\begin{equation}
	\underline{\phi}_0(r) \equiv {\rm const.} \quad \mbox{for \quad $0 \leq r \leq r_0$ 
	\quad  with some $r_0 > 0$,}
	\label{eq3.15}
\end{equation}
\begin{equation}
	\phi_0(r) \geq \underline{\phi}_0(r) \quad \mbox{for all} \ r \geq 0,
	\label{eq3.16}
\end{equation}
\begin{equation}
	r^{N-1}\underline{\phi}_0'(r) + \int^r_0s^{N-1}e^{\underline{\phi}_0(s)}ds \geq 0
	\quad \mbox{for} \ r > 0,
	\label{eq3.17}
\end{equation}
\begin{equation}
	\underline{\phi}_0(r) = \phi(r, \log \alpha) + o(r^{-\lam_0}) 
	\quad \mbox{as} \ r \to \infty 
	\quad \mbox{if} \ N \geq 11
	\label{eq3.18}
\end{equation}
and 
\begin{equation}
	\underline{\phi}_0(r) = \phi(r, \log \alpha) + o(r^{-\lam_0}\log r) 
	\quad \mbox{as} \ r \to \infty
	\quad \mbox{if} \ N = 10.
	\label{eq3.19}
\end{equation}
\end{lemma}

We will give the proof of Lemmas \ref{lem3.3} and \ref{lem3.4} in Section 7 
after proving Theorems \ref{thm1.2} and \ref{thm1.3}.


\section{Boundedness of global solutions} 

In this section, we will show the following proposition, 
which corresponds to the assertion (ii) of Theorems \ref{thm1.2} and \ref{thm1.3}. 

\begin{proposition}\label{prp4.1}
Assume that $u_0(r)$ satisfies $(\ref{eq1.4})$ and, in addition,  
$(\ref{eq1.13})$ if $N \geq 11$ and $(\ref{eq1.15})$ if $N = 10$.
Then the solution $(u, v)$ of $(\ref{eq1.1})$ satisfies 
$$
	\sup_{t > 0}\|u(\cdot, t)\|_{L^{\infty}(\R^N)} < \infty.
$$
\end{proposition}

\begin{proof}[Proof of Proposition \ref{prp4.1}]
First we consider the case $N \geq 11$.
From (\ref{eq1.13}) there exist constants $C_1 > 0$ and $r_1 > 0$ such that 
$$
	r^{2+\lam_0}\left(\frac{2N-4}{r^2}-u_0(r)\right) \geq C_1 
	\quad \mbox{for} \ r \geq r_1.
$$
Then it follows that 
$$
	0 < u_0(r) \leq \frac{2N-4}{r^2} - C_1r^{-2-\lam_0} 
	=\frac{2N-4}{r^2}\left(1-\frac{C_1}{2N-4}r^{-\lam_0}\right)
	\quad \mbox{for} \ r \geq r_1.
$$
Define $\phi_0(r) = \log u_0(r)$. 
Since $\log(1-x) < -x$ for $0 < x < 1$, 
we have 
$$
	\phi_0(r) \leq -2\log r + \log(2N-4) -\frac{C_1}{2N-4}r^{-\lam_0} 
	\quad \mbox{as} \ r \to \infty.
$$
Since $\ell_N(\alpha) \to 0$ as $\alpha \to \infty$ from (\ref{eq3.8}), 
there exists $\bar{\alpha} > 0$ such that 
$\ell_N(\bar{\alpha}) < C_1/(2N-4)$. 
Then, from (\ref{eq3.6}), there exists $r_2 > 0$ such that 
$$
	\phi_0(r) \leq \phi(r, \log \bar{\alpha}) \quad \mbox{for} \ r \geq r_2.
$$
By Lemma \ref{lem3.3} there exists $\hat{\alpha} > 0$ such that 
$$
	\phi_0(r) \leq \phi(r, \log \hat{\alpha}) \quad \mbox{for all} \ r \geq 0.
$$
This implies that $u_0(r) \leq e^{\phi(r, \log \hat{\alpha})}$ for $r \geq 0$. 
Define $w_0(r)$ and $W_0(r)$, respectively, by (\ref{eq1.18}) and 
$$
	W_0(r) = \frac{1}{r^N}\int^r_0 s^{N-1}e^{\phi(s, \log \hat{\alpha})}ds
	\quad \mbox{for} \ r > 0
	\quad \mbox{and} \quad W_0(0) = \lim_{r \to 0}W_0(r) = \frac{\hat{\alpha}}{N}.
$$
Then we have $w_0(r) \leq W_0(r)$ for $r \geq 0$.
Proposition \ref{prp2.4} implies that $W_0 \in C^2[0, \infty)$ satisfies 
(\ref{eq2.4}) with $W_0'(0) = 0$. 
Then $W_0(|x|)$ with $x \in \R^{N+2}$ is a positive radial 
solutions of (\ref{eq2.3}) in $\R^{N+2}$.
Let $w(x, t)$ be a solution of (\ref{eq1.19}).  
Then, by Corollary \ref{cor2.5} we see that $w$ is global and satisfies 
$$
	w(x, t) \leq W_0(|x|) 
	\quad \mbox{for all} \  (x, t) \in (\R^{N+2} \times [0, \infty)).
$$
Thus we obtain 
$\sup_{t > 0}\|w(\cdot, t)\|_{L^{\infty}} \leq \|W_0(\cdot)\|_{L^{\infty}[0, \infty)}$.

Next we consider the case $N = 10$. 
From (\ref{eq1.15}) there exist constants $C_1 > 0$ and $r_1 > 0$ such that 
$$
	\frac{r^{2+\lam_0}}{\log r}\left(\frac{2N-4}{r^2}-u_0(r)\right) \geq C_1 
	\quad \mbox{for} \ r \geq r_1.
$$
Then it follows that 
$$
	0 < u_0(x) \leq \frac{2N-4}{r^2} - C_1r^{-2-\lam_0}\log r 
	=\frac{2N-4}{r^2}\left(1-\frac{C_1}{2N-4}r^{-\lam_0}\log r \right)
$$
for $r \geq r_1$.
Define $\phi_0(r) = \log u_0(r)$. 
Since $\log(1-x) < -x$ for $0 < x < 1$, 
we have 
$$
	\phi_0(r) \leq -2\log r + \log(2N-4) -\frac{C_1}{2N-4}r^{-\lam_0}\log r 
	\quad \mbox{as} \ r \to \infty.
$$
Since $\ell_N(\alpha) \to 0$ as $\alpha \to \infty$ from (\ref{eq3.8}), 
there exists $\bar{\alpha} > 0$ such that 
$\ell_N(\bar{\alpha}) < C_1/(2N-4)$.  
Then, by the same argument as in the case $N \geq 11$, we obtain 
$\sup_{t > 0}\|w(\cdot, t)\|_{L^{\infty}} \leq \|W_0(\cdot)\|_{L^{\infty}[0, \infty)}$.
By Proposition \ref{prp2.7} we conclude that 
$\sup_{t > 0}\|u(\cdot, t)\|_{L^{\infty}(\R^N)} < \infty$ 
in the both cases $N \geq 11$ and $N = 10$.
\end{proof}


\section{Unboundedness of global solutions} 

In this section, we will show the following proposition, 
which corresponds to the assertion (i) of Theorems \ref{thm1.2} and \ref{thm1.3}. 

\begin{proposition}\label{prp5.1}
Assume that $u_0(r)$ satisfies $(\ref{eq1.4})$ and, 
in addition, $(\ref{eq1.12})$ if $N \geq 11$ and $(\ref{eq1.14})$ if $N = 10$.
Then the solution $(u, v)$ of $(\ref{eq1.1})$ satisfies
$$
	\|u(\cdot, t)\|_{L^{\infty}(\R^N)} \to \infty \quad \mbox{as} 
	\ t \to \infty.
$$
\end{proposition}

Theorems \ref{thm1.2} and \ref{thm1.3} follow immediately 
from Propositions \ref{prp4.1} and \ref{prp5.1}.
In order to prove Proposition \ref{prp5.1}, we need some lemmas. 
Throughout this section, we assume that 
$u_0$ in $(\ref{eq1.1})$ satisfies the hypothesis in Proposition \ref{prp5.1}.
Define 
\begin{equation}
	u_{0, m}(r) = \min\{u_0(s): 0 \leq s \leq r\}
	\quad \mbox{for} \ r \geq 0.
	\label{eq5.1}
\end{equation}
It is clear that $u_{0, m}(r) \leq u_0(r)$ for $r \geq 0$ and 
$u_{0, m}(r)$ is nonincreasing for $r > 0$. 
Furthermore, we obtain the following.

\begin{lemma}\label{lem5.2}
The function $u_{0, m}(r)$, defined by $(\ref{eq5.1})$, satisfies 
\begin{equation}
	\lim_{r \to \infty}r^{2+\lam_0}\left(U_{\infty}(r)- u_{0, m}(r)\right)=  0
	\label{eq5.2}
\end{equation}
if $N \geq 11$, and 
\begin{equation}
	\lim_{r \to \infty}\frac{r^{2+\lam_0}}{\log r}
	\left(U_{\infty}(r)- u_{0, m}(r)\right) =  0
	\label{eq5.3}
\end{equation}
if $N \geq 10$. 
\end{lemma}

\begin{proof}
We will show that (\ref{eq5.2}) holds in the cases $N \geq 11$.
Since $u_0$ satisfies (\ref{eq1.12}), for any $\vep > 0$ there exists $r_1 > 0$ such that 
$$
	U_{\infty}(r) - u_0(r) < \frac{\vep}{r^{2+\lam_0}} \quad \mbox{for} \ r \geq r_1,
$$
which implies that 
\begin{equation}
	u_0(r) > U_{\infty}(r) - \frac{\vep}{r^{2+\lam_0}} \quad \mbox{for} \ r \geq r_1.
	\label{eq5.4}
\end{equation}
Recall that $\lam_0 > 0$. Then we may assume that 
\begin{equation}
	U_{\infty}(r) - \frac{\vep}{r^{2+\lam_0}} \  \mbox{is nonincreasing for} \ r \geq r_1.
	\label{eq5.5}
\end{equation}
Since $u_0(r) > 0$ and $u_0(r) \to 0$ as $r \to \infty$, there exists 
$r_2 \geq r_1$ such that, if $r \geq r_2$, then 
$u_{0, m}(r) = u_0(s)$ with some $s \in [r_1, r]$.  
Let $r \geq r_2$. Then it follows from (\ref{eq5.4}) and (\ref{eq5.5}) that 
$$
	u_{0, m}(r) = u_0(s) >  
	U_{\infty}(s) - \frac{\vep}{s^{2+\lam_0}} \geq 
	U_{\infty}(r) - \frac{\vep}{r^{2+\lam_0}}.
$$
From $u_{0, m}(r) \leq u_0(r) \leq U_{\infty}(r)$,
we obtain 
$$
	0 \leq U_{\infty}(r) - u_{0, m}(r) 
	< \frac{\vep}{r^{2+\lam_0}} \quad \mbox{for} \ r \geq r_2,
$$
which implies that (\ref{eq5.2}) holds. 
We obtain (\ref{eq5.3}) by the same argument. 
Hence, we omit the proof.
\end{proof}

For a solution $(u, v)$ of (\ref{eq1.1}), define $w(r, t)$ and $w_0(r)$ by 
(\ref{eq1.16}) and (\ref{eq1.18}), respectively. 
Let $(u_m, v_m)$ be a solution of $(1.1)$ with $u_0 = u_{0, m}$.
Since $u_{0, m}$ satisfies $u_{0, m}(r) < U_{\infty}(r)$ for $r > 0$, 
we see that $(u_m, v_m)$ is global by Theorem A. 
Define $w_m(r, t)$ by 
\begin{equation}
	w_m(r, t) = \frac{1}{r^N}\int^r_0 s^{N-1}u_m(s, t) ds 
	\quad \mbox{for} \ r > 0, \ t \geq 0.
	\label{eq5.6}
\end{equation}
Define $w_{0, m}(r) = w_m(r, 0)$, that is, 
\begin{equation}
	w_{m, 0}(r) = \frac{1}{r^N}\int^r_0 s^{N-1}u_{0, m}(s) ds 
	\quad \mbox{for} \ r > 0 
	\quad \mbox{and} \quad w_{0, m}(0) = \frac{u_{0, m}(0)}{N}.
	\label{eq5.7}
\end{equation}
We obtain the following lemma.

\begin{lemma}\label{lem5.3}
\begin{itemize}
\item[{\rm (i)}] One has $w_m(r, t) \leq w(r, t)$ for $r \geq 0$ and $t \geq 0$. 

\item[{\rm (ii)}] $\|w_m(\cdot, t)\|_{L^{\infty}[0, \infty)} = w_m(0, t)$ for $t \geq 0$. 

\item[{\rm (iii)}] If $\|w_m(\cdot, t)\|_{L^{\infty}[0, \infty)} \to \infty$ as $t \to \infty$, then 
$\|u(\cdot, t)\|_{L^{\infty}(\R^N)} \to \infty$ as $t \to \infty$.
\end{itemize}
\end{lemma}

\begin{proof}
Since $u_{0, m}(r)$ is nonincreasing, 
by Lemma \ref{lem2.6} we have $(w_m)_r(r, t) \leq 0$ for 
$r \geq 0$ and $t \geq 0$, which implies that 
$x\cdot \nabla w_m(|x|, t) \leq 0$. 
Since $u_{0, m}(r) \leq u_0(r)$ for $r \geq 0$,  
we have $w_{m, 0}(r) \leq w_0(r)$ for $r \geq 0$. 
By Proposition \ref{prp2.1}, we obtain $w_m(r, t) \leq w(r, t)$ for $r \geq 0$ and $t >0$. 
Thus (i) holds. 
Since $(w_m)_r(r, t) \leq 0$ for $r \geq 0$, we obtain (ii).  
By Lemma \ref{lem2.6} we have $Nw(0, t) = u(0, t)$ for $t \geq 0$. 
Then, if $\|w_m(\cdot, t)\|_{L^{\infty}[0, \infty)} \to \infty$ as $t \to \infty$, 
by (i) and (ii) we obtain 
$\|u(\cdot, t)\|_{L^{\infty}(\R^N)} \to \infty$ as $t \to \infty$. 
Thus (iii) holds.
\end{proof}

For $\alpha > 0$, let $\phi(r, \log \alpha)$ be the solution of $(\ref{eq2.5})$.
Define $\phi_0(r) = \log u_{0, m}(r)$ for $r \geq 0$, 
where $u_{0, m}$ is defined by (\ref{eq5.1}). 

\begin{lemma}\label{lem5.4}
For any $\alpha > 0$, there exists $r_1 > 0$ such that $\phi_0(r) \geq \phi(r, \log \alpha)$ 
for $r \geq r_1$. 
\end{lemma}

\begin{proof}
First we consider the case $N \geq 11$. 
Take $C > 0$ such that 
\begin{equation}
	C \leq \frac{1}{4}\ell_N(\alpha)(2N-4),
	\label{eq5.8}
\end{equation}
where $\ell_N(\alpha)$ is the constant in Proposition \ref{prp3.1}. 
From (\ref{eq5.2}), there exists $r_0 > 0$ such that 
$$
	r^{2+\lam_0}\left(\frac{2N-4}{r^2} - u_{0, m}(r)\right) \leq C
	\quad \mbox{for} \ r \geq r_0.
$$
Then it follows that 
$$
	u_{0, m}(r) \geq \frac{2N-4}{r^2} - Cr^{-2-\lam_0} 
	\geq \frac{2N-4}{r^2}\left(1-\frac{1}{4}\ell_N(\alpha)r^{-\lam_0}\right)
	\quad \mbox{for} \ r \geq r_0.
$$
Since $\log(1-x) \geq -2x$ for $0 < x < 1/2$, we see that 
$\phi_0(r) = \log u_{0, m}(r)$ satisfies  
$$
	\phi_0(r) \geq -2\log r + \log(2N-4) - \frac{1}{2}\ell_N(\alpha)r^{-\lam_0} 
$$
for sufficiently large $r$.
From (\ref{eq3.6}) there exists $r_1 \geq r_0$ such that 
$\phi_0(r) \geq \phi(r, \log \alpha)$ for $r \geq r_1$.

Next we consider the case $N = 10$. 
Take $C > 0$ such that (\ref{eq5.8}) holds.
From (\ref{eq5.3}), there exists $r_0 > 0$ such that 
$$
	\frac{r^{2+\lam_0}}{\log r}\left(\frac{2N-4}{r^2} - u_{0, m}(r)\right) \leq C
	\quad \mbox{for} \ r \geq r_0.
$$
Then it follows that 
$$
	u_{0, m}(r) \geq \frac{2N-4}{r^2} - Cr^{-2-\lam_0}\log r 
	\geq \frac{2N-4}{r^2}\left(1-\frac{1}{4}\ell_N(\alpha)r^{-\lam_0}\log r \right)
$$
for $r \geq r_0$.
Arguing as above, we obtain 
$$
	\phi_0(r) \geq -2\log r + \log(2N-4) - \frac{1}{2}\ell_N(\alpha)r^{-\lam_0}\log r 
$$
for sufficiently large $r$.
From (\ref{eq3.7}) there exists $r_1 \geq r_0$ such that 
$\phi_0(r) \geq \phi(r, \log \alpha)$ for $r \geq r_1$.
\end{proof}

Take $\alpha > e^{\phi_0(0)}$. Then $\log \alpha > \phi_0(0)$, and hence we have
$$
	\sup_{r \geq 0}(\phi(r, \log \alpha)-\phi_0(r)) > 0.
$$
By Lemmas \ref{lem3.4} and \ref{lem5.4}, there exists $\underline{\phi}_0 \in C^1[0, \infty)$ 
satisfying (\ref{eq3.14})--(\ref{eq3.19}) in Lemma \ref{lem3.4}.
Define $\underline{w}_0(r)$ by 
\begin{equation}
	\underline{w}_0(r) = \frac{1}{r^N}\int^r_0 s^{N-1}
	e^{\underline{\phi}_0(s)}ds
	\quad \mbox{for} \ r > 0
	\quad \mbox{and} \quad 
	\underline{w}_0(0) = \lim_{r \to 0}\underline{w}_0(r).
	\label{eq5.10}
\end{equation}
From (\ref{eq3.16}) we have 
$e^{\underline{\phi}_0(r)} \leq e^{\phi_0(r)} = u_{0, m}(r)$ for $r \geq 0$. 
From (\ref{eq5.7}) we have 
\begin{equation}
	\underline{w}_0(r) \leq w_{0, m}(r) \quad \mbox{for} \ r \geq 0.
	\label{eq5.11}
\end{equation}
Furthermore, $\underline{w}_0$ satisfies the following properties.

\begin{lemma}\label{lem5.5}
Let $\underline{w}_0$ be defined by $(\ref{eq5.10})$.
Then $\underline{w}_0 \in C^2[0, \infty)$ and it satisfies 
$\underline{w}_0'(0) = 0$, 
\begin{equation}
	\underline{w}_0'' + \frac{N+1}{r}\underline{w}_0' 
	+ r\underline{w}_0'\underline{w}_0 + N\underline{w}_0^2 \geq 0 
	\quad \mbox{for}\  r > 0,
	\label{eq5.12}
\end{equation}
\begin{equation}
	\underline{w}_0(r) = \frac{2}{r^2} - \frac{2N-4}{N-2-\lam_0}\ell_N(\alpha)r^{-2-\lam_0} 
	+ o(r^{-2-\lam_0}) 
	\quad \mbox{as} \ r \to \infty 
	\label{eq5.13}
\end{equation}
if $N \geq 11$, and 
\begin{equation}
	\underline{w}_0(r) = \frac{2}{r^2} - \frac{2N-4}{N-2-\lam_0}\ell_N(\alpha)r^{-2-\lam_0}\log r 
	+ o(r^{-2-\lam_0}\log r) 
	\quad \mbox{as} \ r \to \infty 
	\label{eq5.14}
\end{equation}
if $N = 10$.
\end{lemma}

\begin{proof}
Since $\underline{\phi}_0(r)$ is a constant function for $0 \leq r \leq r_0$ 
by (\ref{eq3.15}), we have $\underline{w}_0 \in C^2[0, \infty)$ and $\underline{w}_0'(0) = 0$. 
Define $\Psi(r) = e^{\underline{\phi}_0(r)}$ for $r \geq 0$. 
Then we have $\Psi(r) > 0$ and $\Psi'(r) = \underline{\phi}_0'(r)\Psi(r)$ for $r \geq 0$.
From (\ref{eq5.10}) we have 
\begin{equation}
	r^N\underline{w}_0(r) = \int^r_0 s^{N-1}\Psi(s)ds
	\quad \mbox{for} \ r > 0.
	\label{eq5.15}
\end{equation}
Differentiating the above, we obtain 
\begin{equation}
	N\underline{w}_0(r) + r\underline{w}_0'(r) = \Psi(r) > 0
	\quad \mbox{for} \ r > 0.
	\label{eq5.16}
\end{equation}
Differentiating the above again, we obtain 
\begin{equation}
	r\underline{w}_0''(r) + (N+1)\underline{w}_0'(r) 
	= \underline{\phi}_0'(r)\Psi(r)
	= \underline{\phi}_0'(r)(N\underline{w}_0(r) + r\underline{w}_0'(r)).
	\label{eq5.17}
\end{equation}
From (\ref{eq3.17}) and (\ref{eq5.15}) we obtain 
$r^{N-1}\underline{\phi}_0'(r) + r^{N}\underline{w}_0 \geq 0$, and hence 
$\underline{\phi}_0'(r) \geq - r\underline{w}_0(r)$.
From (\ref{eq5.16}) and (\ref{eq5.17}) it follows that  
$$
	r\underline{w}_0''(r) + (N+1)\underline{w}_0'(r) 
	\geq 
	-r\underline{w}_0(r)(N\underline{w}_0(r) + r\underline{w}_0'(r))
	\quad \mbox{for} \ r > 0.
$$
Thus we obtain (\ref{eq5.12}).

We will show (\ref{eq5.13}) and (\ref{eq5.14}). 
From (\ref{eq3.18}) and (\ref{eq3.19}), we have 
$$
	\underline{\phi}_0(r) = -2\log r +\log(2N-4) -\ell_N(\alpha)r^{-\lam_0} 
	+ o(r^{-\lam_0})
	\quad \mbox{as} \ r \to \infty
$$
if $N \geq 11$ and 
$$
	\underline{\phi}_0(r) = -2\log r +\log(2N-4) -\ell_N(\alpha)r^{-\lam_0}\log r 
	+ o(r^{-\lam_0}\log r)
	\quad \mbox{as} \ r \to \infty.
$$
if $N = 10$. 
Using the same argument as in the proof of Corollary \ref{cor3.2}, we obtain 
$$
	\begin{array}{rcl}
	\dsp e^{\underline{\phi}_0(r)} & = & \dsp 
	\frac{2N-4}{r^2}e^{-\ell_N(\alpha)r^{-\lam_0} + o(r^{-\lam_0})}
	\\[2ex] & = & \dsp
	\frac{2N-4}{r^2}
	\left(1-\ell_N(\alpha)r^{-\lam_0} + o(r^{-\lam_0})\right)
	\\[2ex] 
	& = & \dsp \frac{2N-4}{r^2} -(2N-4)\ell_N(\alpha)r^{-2-\lam_0} + o(r^{-2-\lam_0})
	\end{array}
$$
as $r \to \infty$ if $N \geq 11$ and 
$$
	e^{\underline{\phi}_0(r)} =  
	\frac{2N-4}{r^2} -(2N-4)\ell_N(\alpha)r^{-2-\lam_0}\log r 
	+ o(r^{-2-\lam_0}\log r)
$$
as $r \to \infty$ if $N = 10$. 
By (\ref{eq5.10}) we obtain (\ref{eq5.13}) and (\ref{eq5.14}).
\end{proof}

We are now in a position to prove Proposition \ref{prp5.1}. 

\begin{proof}[Proof of Proposition \ref{prp5.1}] 
Let $(u_m, v_m)$ be a solution of $(1.1)$ with $u_0 = u_{0, m}$, and define 
$w_m$ by (\ref{eq5.6}). 
To prove Proposition \ref{prp5.1} it suffices to show that 
$\|w_m(\cdot, t)\|_{L^{\infty}[0, \infty)} \to \infty$ as $t \to \infty$ by Lemma \ref{lem5.3}.
We will show that, for any large $\alpha > 0$, 
\begin{equation}
	\liminf_{t \to \infty}\|w_m(\cdot, t)\|_{L^{\infty}} \geq \frac{\alpha}{N},
	\label{eq5.18}
\end{equation}
which implies that $\|w_m(\cdot, t)\|_{\infty} \to \infty$ as $t \to \infty$. 

Take any $\alpha > e^{\phi_0(0)}$, and define $\underline{w}_0$ by (\ref{eq5.10}). 
Let $\underline{w}(r, t)$ be a solution of (\ref{eq1.19}) with $w_0 = \underline{w}_0$. 
Recall that we obtain $x \cdot \nabla w_m(|x|, t) \leq 0$ in the proof of Lemma \ref{lem5.3}. 
By Proposition \ref{prp2.1} and (\ref{eq5.11}), we obtain 
\begin{equation}
	\underline{w}(r, t) \leq w_m(r, t) \quad \mbox{for} \ r \geq 0, \ t \geq 0.
	\label{eq5.19}
\end{equation}
From (\ref{eq5.12}) we see that $\underline{w}_0$ is a subsolution of (\ref{eq2.3}).
Then, by Proposition \ref{prp2.2}, $\underline{w}(r, t)$ 
is nondecreasing in $t > 0$ for each fixed $r \geq 0$. 
Define $W(r) = \lim_{t \to \infty}\underline{w}(r, t)$ for $r \geq 0$. 
In the case where $W \not\in L^{\infty}(\R^N)$, 
from (\ref{eq5.19}) we have $\lim_{t \to \infty}\|w_m(\cdot, t)\|_{L^{\infty}[0, \infty)} = \infty$. 

Let us consider the case where $W \in L^{\infty}(\R^N)$. 
In this case, Proposition \ref{prp2.3} implies that 
$W \in C^2(\R^{N+2})$ is a radial solution of (\ref{eq2.3}) in $\R^{N+2}$. 
Define $\bar{\alpha} = W(0)$. Then, by Proposition \ref{prp2.4} we have
\begin{equation}
	W(r) = \frac{1}{r^N}\int^r_0 
	s^{N-1}e^{\phi(s, \log(N\bar{\alpha}))}ds. 
	\label{eq5.20}
\end{equation}
In the case $N \geq 11$, by Corollary \ref{cor3.2} we have
$$
	e^{\phi(r, \log(N\bar{\alpha}))} =
	 \frac{2N-4}{r^2} - (2N-4)\ell_N(N\bar{\alpha})r^{-2-\lam_0} + o(r^{-2-\lam_0})
	\quad \mbox{as} \ r \to \infty.
$$
Then, from (\ref{eq5.20}) we obtain  
\begin{equation}
	W(r) = \frac{2}{r^2} - \frac{2N-4}{N-2-\lam_0}\ell_N(N\bar{\alpha})r^{-2-\lam_0} 
	+ o(r^{-2-\lam_0})
	\quad \mbox{as} \ r \to \infty.
	\label{eq5.21}
\end{equation}
Since $W(r) \geq \underline{w}_0(r)$ for $r \geq 0$, 
from (\ref{eq5.13}) and (\ref{eq5.21}) we have $\ell_N(N\bar{\alpha}) \leq \ell_N(\alpha)$. 
Since $\ell_N(\alpha)$ is decreasing for $\alpha > 0$, 
we have $N\bar{\alpha} \geq \alpha$. 
Thus we obtain 
\begin{equation}
	\liminf_{t \to \infty}\|w_m(\cdot, t)\|_{L^{\infty}[0, \infty)} 
	= W(0) = \bar{\alpha} \geq \frac{\alpha}{N},
	\label{eq5.22}
\end{equation}
hence (\ref{eq5.18}) holds.

In the case $N = 10$, by Corollary \ref{cor3.2} we have 
$$
	e^{\phi(r, \log(N\bar{\alpha}))} =
	 \frac{2N-4}{r^2} - (2N-4)\ell_N(N\bar{\alpha})r^{-2-\lam_0}\log r 
	 + o(r^{-2-\lam_0}\log r)
	\quad \mbox{as} \ r \to \infty.
$$
Thus, from (\ref{eq5.20}) we obtain 
\begin{equation}
	W(r) = \frac{2}{r^2} - \frac{2N-4}{N-2-\lam_0}\ell_N(N\bar{\alpha})r^{-2-\lam_0}\log r 
	+ o(r^{-2-\lam_0}\log r)
	\quad \mbox{as} \ r \to \infty.
	\label{eq5.23}
\end{equation}
Since $W(r) \geq \underline{w}_0(r)$ for $r \geq 0$, 
from (\ref{eq5.14}) and (\ref{eq5.23}) we have $\ell_N(N\bar{\alpha}) \leq \ell_N(\alpha)$, 
and hence $N\bar{\alpha} \geq \alpha$. 
Thus we obtain (\ref{eq5.22}). 
As a consequence, in the both cases $N \geq 11$ and $N = 10$, 
we obtain (\ref{eq5.18}) for any $\alpha > e^{\phi_0(0)}$.  
This implies that 
$\|w_m(\cdot, t)\|_{L^{\infty}[0, \infty)} \to \infty$ as $t \to \infty$.
By Lemma \ref{lem5.3} we conclude that 
$\|u(\cdot, t)\|_{L^{\infty}(\R^N)} \to \infty$ as $t \to \infty$.
\end{proof}


\section{Proof of Proposition 3.1}

\subsection{Preliminaries} 

Let us  consider the linear ordinary differential equation
\begin{equation}
	z'' + C_0z' + c(s)z = 0 \quad \mbox{for} \ s \in \R,
	\label{eq6.1}
\end{equation}
where $C_0 \geq 0$ is a constant and $c \in C(\R)$ satisfies 
\begin{equation}
	|c(s)| \leq C_1 e^{-\delta s} \quad \mbox{for} \ s \geq S_0
	\label{eq6.2}
\end{equation}
with some constants $C_1 > 0$, $\delta > 0$ and $S_0 \in \R$. 
We note that (6.1) can be written as 
\begin{equation}
	(p_0(s)z')' + p_0(s)c(s)z = 0 \quad \mbox{for} \ s \in \R,
	\label{eq6.3}
\end{equation}
where $p_0(s) = e^{C_0s}$. 
We will show the following proposition.

\begin{proposition}\label{prp6.1}
Let $z$ be a solution of $(\ref{eq6.1})$. 
\begin{itemize}
\item[{\rm (i)}] Let $C_0 > 0$. 
Then there exists some $\ell_0 \in \R$ such that 
\begin{equation}
	\lim_{s \to \infty}z(s) = \ell_0.
	\label{eq6.4}
\end{equation}
In the case $\ell_0 = 0$ in $(6.4)$,  
there exists some $\ell_1 \in \R$ such that 
\begin{equation}
	\lim_{s \to \infty}p_0(s)z(s) = \ell_1.
	\label{eq6.5}
\end{equation}
In particular, $z$ satisfies
$$
	\int^{\infty}\frac{dr}{p_0(r)z(r)^2} = \infty.
$$

\item[{\rm (ii)}] Let $C_0 = 0$. 
Then there exists some $\ell_0 \in \R$ such that 
\begin{equation}
	\lim_{s \to \infty}\frac{z(s)}{s} = \ell_0.
	\label{eq6.6}
\end{equation}
In the case $\ell_0 = 0$ in $(6.6)$, there exists some $\ell_1 \in \R$ such that 
\begin{equation}
	\lim_{s \to \infty}z(s) = \ell_1.
	\label{eq6.7}
\end{equation}
In particular, $z$ satisfies
$$
	\int^{\infty}\frac{dr}{z(r)^2} = \infty.
$$
\end{itemize}
\end{proposition}

To prove Proposition \ref{prp6.1}, we need the following lemma.

\begin{lemma}\label{lem6.2}
There exists a constant $S_1 \in \R$ such that 
$(\ref{eq6.1})$ has a positive solution $z_1$ on $[S_1, \infty)$ satisfying 
\begin{equation}
	\frac{1}{2} \leq z_1(s) \leq \frac{3}{2} \quad \mbox{for} \ s \geq S_1
	\quad \mbox{and} \quad 
	\lim_{s \to \infty}z_1(s) = 1.
	\label{eq6.8}
\end{equation}
\end{lemma}

\begin{proof}
First we consider the case $C_0 > 0$. 
We may assume that the constant $\delta > 0$ in (\ref{eq6.2}) satisfies
$\delta < C_0$ without loss of generality. 
Take $S_1 \geq S_0$ such that 
$$
	\frac{C_1}{\delta(C_0-\delta)}e^{-S_1} \leq \frac{1}{3}.
$$
Since $p_0(s)|c(s)| \leq C_1e^{(C_0-\delta)s}$ for $s \geq S_0$, 
we have 
\begin{equation}
	\int^{\infty}_{S_1}\frac{1}{p_0(s)}\int^s_{S_1}p_0(r)|c(r)|drds \leq 
	\frac{C_1}{C_0-\delta}\int^{\infty}_{S_1}e^{-\delta s}ds \leq \frac{1}{3}.
	\label{eq6.9}
\end{equation}
Define the set $Z \subset C[S_1, \infty)$ and the mapping $F: Z \to C[S_1, \infty)$ by
\begin{equation}
	\textstyle
	Z = \{z \in C[S_1, \infty): \frac{1}{2} \leq z(s) \leq \frac{3}{2} \ \mbox{for} \ 
	s \geq S_1\}
	\label{eq6.10}
\end{equation}
and 
$$
	Fz(s) = 1 +\int^{\infty}_s \frac{1}{p_0(\sigma)}
	\int^{\sigma}_{S_1} p_0(r)c(r)z(r)dr d\sigma
	\quad \mbox{for} \ s \geq S_1,
$$
respectively.
Let $z \in Z$. From (\ref{eq6.9}) we have 
$$
	|Fz(s)-1| \leq 
	\int^{\infty}_{s}\frac{1}{p_0(\sigma)}
	\int^{\sigma}_{S_1}p_0(r)|c(r)|z(r)drd\sigma \leq 
	\frac{1}{2} \quad \mbox{for} \ s \geq S_1,
$$
which implies that $Fz \in Z$. 

Define $\|z\| = \sup_{s \geq S_1}|z(s)|$. 
For $z_1, z_2 \in Z$ we have 
$$
	\begin{array}{rcl}
	\|Fz_1 -Fz_2\| & \leq &
	\dsp
	\int^{\infty}_{S_1}\frac{1}{p_0(s)}\int^s_{S_1}p_0(r)|c(r)||z_1(r)-z_2(r)|drds 
	\\[2ex]
	& \leq &
	\dsp
	\int^{\infty}_{S_1}\frac{1}{p_0(s)}\int^s_{S_1}p_0(r)|c(r)|drds\|z_1-z_2\|
	\leq \frac{1}{3}\|z_1-z_2\|.
	\end{array}
$$
Thus $F$ is contractive on $Z$.
Then, by the contractive fixed point theorem, 
there exists $z_1 \in Z$ such that $z_1 = Fz_1$, i.e., 
$$
	z_1(s) = 1 +\int^{\infty}_s \frac{1}{p_0(\sigma)}
	\int^{\sigma}_{S_1} p_0(r)c(r)z_1(r)dr d\sigma
	\quad \mbox{for} \ s \geq S_1.
$$
From this formula it follows that $z_1$ is 
a positive solution of (\ref{eq6.1}) satisfying (\ref{eq6.8}). 

Next we consider the case $C_0 = 0$. 
In this case, $p_0(s) \equiv 1$. Take $S_1 \geq S_0$ such that 
$$
	\frac{C_1}{\delta^2}e^{-S_1} \leq \frac{1}{3}.
$$
Since $p_0(s)|c(s)| = |c(s)| \leq C_1e^{-\delta s}$ for $s \geq S_0$, 
we have 
$$
	\int^{\infty}_{S_1}\frac{1}{p_0(s)}\int^{\infty}_s p_0(r)|c(r)|drds \leq 
	\frac{C_1}{\delta}\int^{\infty}_{S_1}e^{-\delta s}ds \leq \frac{1}{3}.
$$
Define the set $Z \subset C[S_1, \infty)$ and the mapping $F: Z \to C[S_1, \infty)$ 
by (\ref{eq6.10}) and 
$$
	Fz(s) = 1 -\int^{\infty}_s \frac{1}{p_0(\sigma)}
	\int^{\infty}_{\sigma} p_0(r)c(r)z(r)drd\sigma
	\quad \mbox{for} \ s \geq S_1,
$$
respectively.
Then by the same argument as above, we see that $FZ \subset Z$ and 
$F$ is contractive on $Z$. 
Then, by the contractive fixed point theorem, we obtain 
a positive solution of (\ref{eq6.1}) satisfying (\ref{eq6.8}).
\end{proof}

\begin{proof}[Proof of Proposition \ref{prp6.1}]  
First we consider the case $C_0 > 0$. 
In this case, define $z_2(s)$ by 
$$
	z_2(s) = z_1(s)\int^{\infty}_s \frac{dr}{p_0(r)z_1(r)^2}
	\quad \mbox{for} \ s \geq S_1.
$$
By a direct calculation, $z_2$ satisfies (\ref{eq6.1}) for $s \geq S_1$.
Recall that $p_0(s) = e^{C_0 s}$. 
By  L'Hospital's rule, we have 
$$
	\lim_{s \to \infty}\frac{1}{e^{-C_0 s}}\int^{\infty}_s \frac{dr}{p_0(r)z_1(r)^2} 
	= \lim_{s \to \infty}\frac{1}{C_0}\frac{e^{C_0 s}}{p_0(s)z_1(s)^2} = \frac{1}{C_0}.
$$	
Then we obtain 
$$
	\lim_{s \to \infty}p_0(s)z_2(s) = \lim_{s \to \infty}
	\frac{z_1(s)}{e^{-C_0 s}}\int^{\infty}_s \frac{dr}{p_0(r)z_1(r)^2} = \frac{1}{C_0}.
$$
Let $z$ be a solution of (\ref{eq6.1}). 
Since $z_1$ and $z_2$ are linearly independent, there exist constants $c_1$ and $c_2$ such that 
\begin{equation}
	z(s) = c_1z_1(s) + c_2 z_2(s).
	\label{eq6.11}
\end{equation}
Since $z_1(s) \to 1$ and $p_0(s)z_2(s) \to \frac{1}{C_0}$ as $s \to \infty$, 
we have $z(s) \to c_1$ as $s \to \infty$. 
Hence, (\ref{eq6.4}) holds with $\ell_0 = c_1$. 
If $c_1 = \ell_0 = 0$, then (\ref{eq6.5}) holds with $\ell_1 = c_2/C_0$.

Next we consider the case $C_0 = 0$. 
In this case, define $z_2(s)$ by 
$$
	z_2(s) = z_1(s)\int^{s}_{S_1} \frac{dr}{z_1(r)^2}
	\quad \mbox{for} \ s \geq S_1.
$$
By a direct calculation, $z_2$ satisfies (\ref{eq6.1}) for $s \geq S_1$.
By L'Hospital's rule, we have 
$$
	\lim_{s \to \infty}\frac{1}{s}\int^{s}_{S_1} \frac{dr}{z_1(r)^2} 
	= \lim_{s \to \infty}\frac{1}{z_1(s)^2} = 1.
$$	
Then we obtain 
$$
	\lim_{s \to \infty}\frac{z_2(s)}{s} = \lim_{s \to \infty}
	\frac{z_1(s)}{s}\int^{\infty}_s \frac{dr}{z_1(r)^2} = 1.
$$
Let $z$ be a solution of (\ref{eq6.1}). 
Since $z_1$ and $z_2$ are linearly independent, there exist constants $c_1$ and $c_2$ such that 
(\ref{eq6.11}) holds.
Since $z_1(s) \to 1$ and $\frac{z_2(s)}{s} \to 1$ as $s \to \infty$, 
we have $\frac{z(s)}{s} \to c_2$ as $s \to \infty$. 
Hence, (\ref{eq6.6}) holds with $\ell_0 = c_2$. 
If $c_2 = \ell_0 = 0$, then (\ref{eq6.7}) holds with $\ell_1 = c_1$.
\end{proof}

We recall the result by \cite[Proposition A.1]{Naia}. 

\begin{proposition}\label{prp6.3}
Let us consider two differential equations 
\begin{equation}
	(p(r)u')' + q(r)u = 0 \quad \mbox{for} \ r > 0,
	\label{eq6.12}
\end{equation}
\begin{equation}
	(p(r)v')' + Q(r)v = 0 \quad \mbox{for} \ r > 0,
	\label{eq6.13}
\end{equation}
where $p, q, Q \in C[0, \infty)$ satisfy $p(r) > 0$ for $r \in (0, \infty)$ and 
$q(r) \leq Q(r)$ for $r \in [0, \infty)$ with $q \not\equiv Q$. 
Assume that $(\ref{eq6.12})$ has a solution $u(r) > 0$ for $r \geq 0$ such that 
$u'(0) = 0$ and 
$$
	\int^{\infty}_1\frac{ds}{p(s)u(s)^2} = \infty.
$$
Let $v$ be a solution of $(\ref{eq6.13})$ satisfying $v(0) > 0$ and $v'(0) = 0$. 
Then $v$ has at least one zero in $(0, \infty)$. 
\end{proposition}


\subsection{Proof of Proposition \ref{prp3.1}} 

By the change of the variable
\begin{equation}
	w(s, \log \alpha) = \phi(r, \log \alpha)- \phi_{\infty}(r) \quad \mbox{with} 
	\ s = \log r,
	\label{eq6.14}
\end{equation}
the function $w(s) = w(s, \log \alpha)$ satisfies
\begin{equation}
	w'' + (N-2)w' + 2(N-2)(e^w -1) = 0 \quad \mbox{for} \  -\infty < s < \infty.
	\label{eq6.15}
\end{equation}
From (\ref{eq3.3}) and (\ref{eq3.5}) we have $\lim_{s \to \infty}w(s, \log \alpha) = 0$ and
\begin{equation}
	w(s, \log \alpha) < w(s, \log \beta) < 0 \quad \mbox{for} \ s \in \R 
	\quad \mbox{if} \ 0 < \alpha < \beta.
	\label{eq6.16}
\end{equation}
Then (\ref{eq6.15}) can be written as 
\begin{equation}
	w'' + (N-2)w' + 2(N-2)h(w)w = 0 \quad \mbox{for} \  -\infty < s < \infty.
	\label{eq6.17}
\end{equation}
where 
\begin{equation}
	h(w) = \frac{e^w -1}{w}.
	\label{eq6.18}
\end{equation}
Let $N \geq 10$, and let $\delta \in (0,1)$. 
Define
\begin{equation}
	\lam_0(\delta) = \frac{N-2-\sqrt{(N-2)(N-2-8\delta)}}{2}
	\label{eq6.19}
\end{equation}
and 
$$
	\lam_1(\delta) = \frac{N-2+\sqrt{(N-2)(N-2-8\delta)}}{2}.
$$
Note that $0 < \lam_0(\delta) < \lam_1(\delta)$ and  $-\lam_0(\delta)$ and $-\lam_1(\delta)$ 
are roots of a polynomial 
\begin{equation}
	\lam^2 +(N-2)\lam + 2\delta(N-2) = 0.
	\label{eq6.20}
\end{equation}
We first show the following lemma.

\begin{lemma}\label{lem6.4}
Let $w(s, \log \alpha)$ be defined by $(\ref{eq6.14})$. 
For $\delta \in (0, 1)$, define $\lam_0(\delta)$ by $(\ref{eq6.19})$.
Then $w(s, \log \alpha)$ satisfies
\begin{equation}
	|w(s, \log \alpha)| = O(e^{-\lam_0(\delta)s}) \quad \mbox{as} \ s \to \infty.
	\label{eq6.21}
\end{equation}
\end{lemma}

\begin{proof} 
For simplicity, we denote $w(s) = w(s, \log \alpha)$. 
Note that $h(w)$, defined by (\ref{eq6.18}), 
is increasing in $w \in \R$, $h(w) \to 1$ as $w \to 0$. 
Since $w(s) < 0$ for $s \in \R$ and $w(s) \to 0$ as $s \to \infty$, 
there exists $s_1 \in \R$ such that 
$\delta < h(w(s)) < 1$ for $s \geq s_1$. 
From (\ref{eq6.17}) we see that $w$ satisfies 
\begin{equation}
	w'' + (N-2)w' + 2(N-2)\delta w > 0 \quad \mbox{for} \  s \geq s_1.
	\label{eq6.22}
\end{equation}
Since $-\lam_0(\delta)$ and $-\lam_1(\delta)$ are roots of the polynomial (\ref{eq6.20}),
the inequality (\ref{eq6.22}) can be written as 
$$
	(w'+\lam_0(\delta)w)' +\lam_1(\delta)(w'+\lam_0(\delta)w) > 0 
	\quad \mbox{for} \  s \geq s_1,
$$
which implies that $(e^{\lam_1(\delta)s}(w'(s)+\lam_0(\delta)w(s)))' > 0$ for $s \geq s_1$. 
Thus we obtain 
$$
	w'(s)+\lam_0(\delta)w(s) \geq Ce^{-\lam_1(\delta)s}
	\quad \mbox{for} \  s \geq s_1 
$$
with some constant $C \in \R$, which implies that 
$$
	(e^{\lam_0(\delta)s}w(s))' \geq  Ce^{-(\lam_1(\delta)-\lam_0(\delta))s}
	\quad \mbox{for} \  s \geq s_1. 
$$
Integrating the above on $[s_1, s]$, we obtain 
$w(s) \geq C_1e^{-\lam_0(\delta)s} + C_2e^{-\lam_1(\delta)s}$ 
for $s \geq s_1$ with some constants $C_1 = e^{\lam_0(\delta)s_1}w(s_1) < 0$ and $C_2 \in \R$. 
Since $0 < \lam_0(\delta) < \lam_1(\delta)$ and $w(s) < 0$, we obtain (\ref{eq6.21}).
\end{proof}

Next we show the following two lemmas.

\begin{lemma}\label{lem6.5}
The function $w(s, \log \alpha)$, defined by $(6.14)$, satisfies  
\begin{equation}
	w(s, \log \alpha) = -\ell_N(\alpha)e^{-\lam_0 s} + o(e^{-\lam_0 s})
	\quad \mbox{as} \ s \to \infty 
	\label{eq6.23}
\end{equation}
if $N \geq 11$, and 
\begin{equation}
	w(s, \log \alpha) = -\ell_N(\alpha)se^{-\lam_0 s} + o(se^{-\lam_0 s})
	\quad \mbox{as} \ s \to \infty 
	\label{eq6.24}
\end{equation}
if $N \geq 10$, 
where $\lam_0$ is defined by $(\ref{eq1.6})$ and  $\ell_N(\alpha)$ is a nonnegative constant.
\end{lemma}

From (\ref{eq6.16}) we have $\ell_N(\alpha_1) \geq \ell_N(\alpha_2)$ 
if $0 < \alpha_1 < \alpha_2$ in Lemma \ref{lem6.5}.  
If $\ell_N(\alpha_1) = \ell_N(\alpha_2)$ we obtain the following.

\begin{lemma}\label{lem6.6}
Let $N \geq 10$, and let $\ell_N(\alpha)$ be the constant in Lemma $\ref{lem6.5}$. 
Assume that $\ell_N(\alpha_1) = \ell_N(\alpha_2)$ with some $0 < \alpha_1 < \alpha_2$.
Define $\eta(s) = w(s, \log \alpha_2) - w(s, \log \alpha_1)$.  
Then $\eta$ satisfies
\begin{equation}
	\int^{\infty}\frac{ds}{e^{(N-2)s}\eta(s)^2} = \infty.
	\label{eq6.25}
\end{equation}
\end{lemma}

\begin{proof}[Proof of Lemma \ref{lem6.5}] 
Recall that $w(s) = w(s, \log \alpha) < 0$ for $s \in \R$ and satisfies (\ref{eq6.17}). 
Note that $\lam_0$, defined by (\ref{eq1.6}), is a root of a polynomial
$$
	\lam^2 -(N-2)\lam + 2(N-2) = 0.
$$
By the change of variable $z(s) = e^{\lam_0 s}w(s)$, we have
\begin{equation}
	z'' + C_0z' + 2(N-2)(h(w)-1)z = 0 \quad \mbox{for} \  -\infty < s < \infty,
	\label{eq6.26}
\end{equation}
where $C_0 = \sqrt{(N-2)(N-10)}$. 
We see that 
$$
	0 \geq h(w)-1 = \frac{e^w-1-w}{w} \geq \frac{1}{2}w 
	\quad \mbox{for} \ w \leq 0.
$$
By Lemma \ref{lem6.4} we have $|h(w(s))-1| = O(e^{-\lam_0(\delta) s})$ as $s \to \infty$. 

Let $N \geq 11$. Then $C_0 > 0$. 
Applying Proposition \ref{prp6.1} for the solution $z$ of (\ref{eq6.26}), we have 
$$
	z(s) = \ell + o(1) \quad \mbox{as} \ s \to \infty
$$
with some $\ell \in \R$. 
Since $z(s) < 0$ for $s \in \R$, we have $\ell \leq 0$. 
Then we obtain (\ref{eq6.23}) with $\ell_N(\alpha) = -\ell \geq 0$.     

Let $N = 10$. Then $C_0 = 0$. 
By Proposition \ref{prp6.1} we have 
$$
	z(s) = \ell s + o(s) \quad \mbox{as} \ s \to \infty
$$
with some $\ell \in \R$. 
Since $z(s) < 0$ for $s \in \R$, we have $\ell \leq 0$. 
Then we obtain (\ref{eq6.24}) with $\ell_N(\alpha) = -\ell \geq 0$.   
\end{proof}

\begin{proof}[Proof of Lemma \ref{lem6.6}] 
Recall that $w(s, \log \alpha)$ satisfies (\ref{eq6.15}). 
Then $\eta(s)$ satisfies 
$$
	\eta'' + (N-2)\eta' + 2(N-2)e^{\theta(s)}\eta = 0 
	\quad \mbox{for} \  -\infty < s < \infty,
$$
where 
\begin{equation}
	w(s, \log \alpha_1) < \theta(s) < w(s, \log \alpha_2) < 0.
	\label{eq6.27}
\end{equation}
Since $\ell_N(\alpha_1) = \ell_N(\alpha_2)$,  by Lemma 6.5 we have
\begin{equation}
	\eta(s) = \left\{
	\begin{array}{ll} 
	o(e^{-\lam_0 s}) \quad & \mbox{if} \ N \geq 11,
	\\[1ex]
	o(se^{-\lam_0 s}) \quad & \mbox{if} \ N = 10.
	\end{array}
	\right.
	\label{eq6.28}
\end{equation}
By the change of the variable $z(s) = e^{\lam_0 s}\eta(s)$, we have
\begin{equation}
	z'' + C_0z' + 2(N-2)(e^{\theta(s)}-1)z = 0 
	\quad \mbox{for} \  -\infty < s < \infty,
	\label{eq6.29}
\end{equation}
where $C_0 = \sqrt{(N-2)(N-10)}$. 
We note that 
$x \leq e^{x}-1 \leq 0$ for $x \leq 0$.
Then, by Lemma~\ref{lem6.5} and (\ref{eq6.27}), we have, as $s \to \infty$,  
$$
	|e^{\theta(s)}-1| \leq |\theta(s)| = 
	\left\{
	\begin{array}{ll}
	O(e^{-\lam_0 s}) \quad & {if} \ N \geq 11, 
	\\[1ex]
	O(s e^{-\lam_0 s}) \quad & {if} \ N = 10. 
	\end{array}
	\right.
$$
Let $p_0(s) = e^{C_0s} = e^{\sqrt{(N-2)(N-10)} s}$. 
Then (\ref{eq6.29}) can be written as (\ref{eq6.3}) with $c(s) = 2(N-2)(e^{\theta(s)}-1)$, 
and the function $c(s)$ 
satisfies $|c(s)| = O(e^{-\delta s})$ as $s \to \infty$ with $0 < \delta < \lam_0$. 

Let $N \geq 11$. From (\ref{eq6.28}) we have $z(s) \to 0$ as $s \to \infty$. 
By Proposition \ref{prp6.1} we have 
$$
	\int^{\infty}\frac{ds}{p_0(s)z(s)^2} = \infty,
$$
which implies (\ref{eq6.25}).
Let $N = 10$. From (\ref{eq6.28}) we have $z(s)/s \to 0$ as $s \to \infty$. 
By Proposition~\ref{prp6.1} we have 
$$
	\int^{\infty}\frac{ds}{z(s)^2} = \infty.
$$
Since $\lam_0 = (N-2)/2$ if $N = 10$, we obtain (\ref{eq6.25}). 
\end{proof}

\begin{proof}[Proof of Proposition \ref{prp3.1}] 
Define $w(s, \log \alpha)$ by (\ref{eq6.14}). 
By Lemma \ref{lem6.5} we obtain (\ref{eq6.23}) and (\ref{eq6.24}) 
when $N \geq 11$ and $N = 10$, respectively.
Then we obtain (\ref{eq3.6}) if $N \geq 11$ and (\ref{eq3.7}) if $N = 10$. 
From (\ref{eq6.16}) we have $\ell_N(\alpha_1) \geq \ell_N(\alpha_2)$ if $0 < \alpha_1 < \alpha_2$.
We will show that 
\begin{equation}
	\ell_N(\alpha_1) > \ell_N(\alpha_2) \quad \mbox{for any} \ 0 < \alpha_1 < \alpha_2.
	\label{eq6.30}
\end{equation}
Assume by contradiction that there exists $0 <  \alpha_1 < \alpha_2$ such that 
$\ell_N(\alpha_1) = \ell_N(\alpha_2)$. 
Take any $\alpha_3 > \alpha_2$, and define 
$$
	u(r) = \phi(r, \log \alpha_2) - \phi(r, \log \alpha_1)
	\quad \mbox{and} \quad  
	v(r) = \phi(r, \log \alpha_3) - \phi(r, \log \alpha_2)
$$
for $r \geq 0$.
Then $u(r)$ is positive on $[0, \infty)$ and, by the mean value theorem, $u$ satisfies
$$
	u'' + \frac{N-1}{r}u' + e^{\theta(r)}u = 0 \quad \mbox{for} \ r > 0,
	\quad u(0) > 0 \quad \mbox{and} \quad u'(0) = 0,
$$
where $\theta$ satisfies $\phi(r, \log \alpha_1) < \theta(r) < \phi(r, \log \alpha_2)$ 
for $r \geq 0$.   
Similarly, $v(r)$ is positive on $[0, \infty)$ and satisfies 
$$
	v'' + \frac{N-1}{r}v' + e^{\Theta(r)}v = 0 \quad \mbox{for} \ r > 0,
	\quad v(0) > 0 \quad \mbox{and} \quad v'(0) = 0,
$$
where $\Theta$ satisfies $\phi(r, \log \alpha_2) < \Theta(r) < \phi(r, \log \alpha_3)$ 
for $r \geq 0$.   
Then we have $e^{\theta(r)} < e^{\Theta(r)}$ for $r \geq 0$. 

Define $\eta(s) = w(s, \log \alpha_2) - w(s, \log \alpha_1)$.
Then $\eta(s) = u(r)$ with $s = \log r$. 
By Lemma~\ref{lem6.6} we have (\ref{eq6.25}), which implies that 
$$
	\int^{\infty}\frac{dr}{r^{N-1}u(r)^2} = \infty.
$$
Applying Proposition \ref{prp6.3} with $p(r) = r^{N-1}$, $q(r) = e^{\theta(r)}$ and 
$Q(r) = e^{\Theta(r)}$, we see that 
$v(r)$ has at least one zero in $(0, \infty)$. 
This is a contradiction. 
Thus we obtain (\ref{eq6.30}).

By Lemma \ref{lem6.5} we have $\ell_N(\alpha) \geq 0$ 
for all $\alpha > 0$ in (\ref{eq6.23}) and (\ref{eq6.24}). 
We will show that $\ell_N(\alpha) > 0$ for all $\alpha > 0$. 
Assume by contradiction that $\ell_N(\alpha_0) = 0$ with some $\alpha_0 > 0$. 
Then, from (\ref{eq6.30}) we obtain $\ell_N(\alpha) < 0$ if $\alpha > \alpha_0$, which is a contradiction. 
Thus we obtain $\ell_N(\alpha) > 0$ for all $\alpha > 0$.

In the case $N \geq 11$, from (\ref{eq3.6}) with $\alpha = 1$, we have 
$$
	\phi(r, 0) = \phi_{\infty}(r) -\ell_N(1)r^{-\lam_0} + o(r^{-\lam_0}) 
	\quad \mbox{as} \ r \to \infty.
$$
Using (\ref{eq3.1}) and the property 
that $\phi_{\infty}(\sqrt{\alpha}r) + \log \alpha = \phi_{\infty}(r)$, we have 
$$
	\begin{array}{rcl}
	\phi(r, \log \alpha) & = & \phi(\sqrt{\alpha}r, 0) + \log \alpha 
	\\[2ex] 
	& = & 	 
	\phi_{\infty}(\sqrt{\alpha}r) + \log \alpha -\ell_N(1)(\sqrt{\alpha}r)^{-\lam_0} 
	+ o(\sqrt{\alpha}r^{-\lam_0})
	\\[2ex]
	& = &  \phi_{\infty}(r) -\ell_N(1)\alpha^{-\lam_0/2}r^{-\lam_0} + o(\lam^{-\lam_0})
	\quad \mbox{as} \ r \to \infty.
	\end{array}
$$
Thus we obtain (\ref{eq3.8}).
In the case $N = 10$, by the same argument we obtain 
$$
	\phi(r, \log \alpha) = 
	\phi_{\infty}(r) - \ell_N(1)\alpha^{-\lam_0/2}r^{-\lam_0}\log r 
	+o(r^{-\lam_0}\log r)
	\quad \mbox{as} \ r \to \infty,
$$
and hence (\ref{eq3.8}) holds.
\end{proof}


\section{Proof of Lemmas \ref{lem3.3} and \ref{lem3.4}}

First we give the proof of Lemma \ref{lem3.3}.

\begin{proof}[Proof of Lemma \ref{lem3.3}] 
Take $\bar{\alpha} > 0$ such that $\log \bar{\alpha} > \phi_0(0)$.
Since $\phi(0, \log \bar{\alpha}) = \log \bar{\alpha} > \phi_0(0)$, 
there exists $r_2 > 0$ satisfying 
\begin{equation}
	\phi(r, \log \bar{\alpha}) > \phi_0(r) \quad \mbox{for} \ 0  \leq r \leq r_2.
	\label{eq7.1}
\end{equation}
Recall that $\phi(r, \log \alpha)$ is increasing in $\alpha > 0$ and (\ref{eq3.4}) holds.
Then there exists $\hat{\alpha} \geq \max\{\bar{\alpha}, \alpha_1\}$ such that 
\begin{equation}
	\phi(r, \log\hat{\alpha}) > \phi_0(r) \quad \mbox{for} \ r_2 \leq r \leq r_1,
	\label{eq7.2}
\end{equation}
where $\alpha_1$ and $r_1$ are constants in (\ref{eq3.13}). 
Combining (\ref{eq3.13}), (\ref{eq7.1}) and (\ref{eq7.2}), 
we obtain $\phi(r, \log\hat{\alpha}) \geq \phi_0(r)$ for $r \geq 0$.
\end{proof}

In the proof of Lemma \ref{lem3.4}, for simplicity 
we denote by $\phi(r, \log \alpha) = \phi_{\alpha}(r)$. 
Define 
\begin{equation}
	\psi(r) = r^{-(N-2)/2} \quad \mbox{for} \ r > 0. 
	\label{eq7.3}
\end{equation}
By a direct calculation, $\psi(|x|)$ with $x \in \R^N$ satisfies
$$
	\Delta \psi + \frac{(N-2)^2}{4|x|^2}\psi = 0 
	\quad \mbox{for} \ x \in \R^N\setminus\{0\}.
$$
Since $e^{\phi_{\infty}(|x|)} = \frac{2N-4}{|x|^2}$ and 
$2N-4 \leq \frac{(N-2)^2}{4}$ if $N \geq 10$, we have 
\begin{equation}
	\Delta \psi + e^{\phi_{\infty}(|x|)}\psi \leq 0
	\quad \mbox{for} \ x \in \R^N\setminus\{0\}.
	\label{eq7.4}
\end{equation}
We show the following lemma.

\begin{lemma}\label{lem7.1}
Let $N \geq 10$.  
For a constant $C > 0$, define
$w(r) = \phi_{\alpha}(r) -C\psi(r)$ for $r > 0$, 
where $\psi$ is defined by $(\ref{eq7.3})$. 
Then $w(|x|)$ with $x \in \R^N$ satisfies
$$
	\Delta w + e^w > 0 \quad \mbox{for} \ |x| > 0.
$$
\end{lemma}

\begin{proof}
From (\ref{eq7.4}) we have 
\begin{equation}
	\Delta w + e^w = \Delta \phi_{\alpha} - C\Delta \psi + e^w
	\geq - e^{\phi_{\alpha}} + Ce^{\phi_{\infty}}\psi + e^w.
	\label{eq7.5}
\end{equation}
By the mean value theorem and $\phi_{\infty}(r) > \phi_{\alpha}(r)$ for $r > 0$, 
we obtain 
\begin{equation}
	e^{\phi_{\alpha}(r)} -e^{w(r)} = 
	e^{\phi_{\alpha}(r)} -e^{\phi_{\alpha}(r) - C\psi(r)} 
	< e^{\phi_{\alpha}(r)}C\psi(r) < Ce^{\phi_{\infty}(r)}\psi(r).
	\label{eq7.6}
\end{equation}
From (\ref{eq7.5}) and (\ref{eq7.6}) we obtain $\Delta w + e^w > 0$ for $|x| > 0$.
\end{proof}

\begin{proof}[Proof of Lemma \ref{lem3.4}] 
Let $\psi$ be defined by (\ref{eq7.3}). 
Observe that 
$$
	\frac{\phi_{\alpha}(r) -\phi_0(r)}{\psi(r)} \to 0 
	\quad \mbox{as} \ r \to 0.
$$
From (\ref{eq3.14}) we have 
$$
	\frac{\phi_{\alpha}(r) -\phi_0(r)}{\psi(r)} \leq 0 
	\quad \mbox{for} \ r \geq r_1
	\quad \mbox{and} \quad 
	\sup_{r > 0}\frac{\phi_{\alpha}(r) -\phi_0(r)}{\psi(r)} > 0. 
$$
Then there exists a constant $C_0 > 0$ such that  
$$
	\sup_{r > 0}\frac{\phi_{\alpha}(r) -\phi_0(r)}{\psi(r)} = C_0.
$$
Define $w(r) = \phi_{\alpha}(r)-C_0\psi(r)$ for $r > 0$. 
Then 
\begin{equation}
	w(r) \leq \phi_0(r) \quad \mbox{for} \ r > 0.
	\label{eq7.7}
\end{equation}
Since $w(r) \to -\infty$ as $r \to 0$ and $r \to \infty$, 
there exists $r_2 > 0$ such that 
$w(r_2) = \max_{r > 0}w(r)$. 
Define
\begin{equation}
	\underline{\phi}_0(r) = \left\{
	\begin{array}{ll}
	w(r) \quad &\mbox{if} \ r \geq r_2,
	\\[1ex]
	w(r_2) \quad &\mbox{if} \ 0 \leq r < r_2.
	\end{array}
	\right.
	\label{eq7.8}
\end{equation}
Then $\underline{\phi}_0 \in C^1[0, \infty)$ and 
\begin{equation}
	\underline{\phi}_0(r) \geq w(r) \quad \mbox{for} \ r \geq 0.
	\label{eq7.9}
\end{equation}
Since $\phi_0(r)$ is nonincreasing, we have 
$$
	\phi_0(r) \geq \phi_0(r_2) \geq w(r_2) = \underline{\phi}_0(r)
	\quad \mbox{for} \ 0 \leq r \leq r_2.
$$
From (\ref{eq7.7}) and (\ref{eq7.8}) we obtain 
$\underline{\phi}_0(r) \leq \phi_0(r)$ for $r \geq r_2$, and hence (\ref{eq3.16}) holds. 
By the definition of $\underline{\phi}_0$ in (\ref{eq7.8}), 
it is clear that (\ref{eq3.15}) holds. 
Note that $\lam_0 < (N-2)/2$ if $N \geq 11$ and $\lam_0 = (N-2)/2$ if $N = 10$.
Then we obtain, as $r \to \infty$, 
$$
	\psi(r) = \left\{
	\begin{array}{ll}
	o(r^{-\lam_0}) \quad & \mbox{if} \ N \geq 11,
	\\[1ex]
	o(r^{-\lam_0}\log r) \quad & \mbox{if} \ N = 10.
	\end{array}
	\right.
$$
Thus we obtain (\ref{eq3.18}) and (\ref{eq3.19}). 

We will show that (\ref{eq3.17}) holds.
It is clear that $r^{N-1}\underline{\phi}_0'(r) \equiv 0$ for $0 \leq r \leq r_2$. 
Then we obtain 
\begin{equation}
	r^{N-1}\underline{\phi}_0'(r) 
	+ \int^r_0 s^{N-1}e^{\underline{\phi}_0(s)}ds > 0
	\quad \mbox{for} \ 0 < r \leq r_2.
	\label{eq7.10}
\end{equation}
By Lemma \ref{lem7.1}, $w(r)$ satisfies 
\begin{equation}
	(r^{N-1}w'(r)) + r^{N-1}e^{w(r)} > 0 \quad  \mbox{for} \ r > 0.
	\label{eq7.11}
\end{equation}
Note that $r^{N-1}w'(r) \to 0$ as $r \to 0$. 
Then, integrating the both sides of (\ref{eq7.11}) on $[\rho, r]$ with $\rho > 0$, 
and letting $\rho \to 0$, we obtain 
$$
	r^{N-1}w'(r) + \int^r_0 s^{N-1}e^{w(s)}ds > 0. 
$$
From (\ref{eq7.9}) we obtain 
$$
	r^{N-1}\underline{\phi}_0'(r) 
	+ \int^r_0 s^{N-1}e^{\underline{\phi}_0(s)}ds > 0
	\quad \mbox{for} \ r > r_2.
$$
From (\ref{eq7.10}) we obtain (\ref{eq3.17}).
\end{proof}


\bigskip

\noindent
{\bf Acknowledgement. } 

The first author was supported by JSPS KAKENHI Grant Number JP26K06881 
and the second author was supported by 
JSPS KAKENHI Grant Number JP23K03190. 
This work was also supported by Research Institute for Mathematical Sciences, a Joint
Usage/Research Center located in Kyoto University.
\bigskip



\end{document}